\documentclass[12pt]{amsart}

\usepackage{amssymb, amsmath, amsthm}
\usepackage[margin=1in]{geometry}
\usepackage{tikz}

\allowdisplaybreaks

\numberwithin{equation}{section}

\theoremstyle{plain}
\newtheorem{thm}{Theorem}[section]
\newtheorem{prop}[thm]{Proposition}
\newtheorem{lem}[thm]{Lemma}
\newtheorem{cor}[thm]{Corollary}
\newtheorem*{referthmA}{Theorem A}

\theoremstyle{definition}
\newtheorem{defn}[thm]{Definition}

\newtheorem{notation}[thm]{Notation}

\newcommand{\ichi}{\mathbf{1}}

\newcommand{\N}{\mathbb{N}}
\newcommand{\R}{\mathbb{R}}

\newcommand{\Z}{\mathbb{Z}}

\newcommand{\calS}{\mathcal{S}}

\newcommand{\dotS}{\Dot{S}}
\newcommand{\supp}{\mathrm{supp}\, }

\newcommand{\RomI}{\mathrm{I}}
\newcommand{\II}{\mathrm{I}\hspace{-0.5pt}\mathrm{I}}
\newcommand{\III}{\mathrm{I}\hspace{-0.5pt}\mathrm{I}\hspace{-0.5pt}\mathrm{I}}

\newcommand{\IV}{\mathrm{I}\hspace{-0.5pt}\mathrm{V}}
\newcommand{\RomV}{\mathrm{V}}
\newcommand{\VI}{\mathrm{V}\hspace{-0.5pt}\mathrm{I}}
\newcommand{\grad}{\textrm{grad}}
\newcommand{\sign}{\mathrm{sign}\,}
\newcommand{\Hess}{\mathrm{Hess}}

\begin{document}

\title[ Bilinear oscillatory Fourier multipliers ]
{ Bilinear oscillatory Fourier multipliers } 

\author[T. Kato]{Tomoya Kato}
\author[A. Miyachi]{Akihiko Miyachi}
\author[N. Shida]{Naoto Shida}
\author[N. Tomita]{Naohito Tomita}

\address[T. Kato]
{Faculty of Engineering, 
Gifu University, 
Gifu 501-1193, Japan}

\address[A. Miyachi]
{Department of Mathematics, 
Tokyo Woman's Christian University, 
Zempukuji, Suginami-ku, Tokyo 167-8585, Japan}

\address[N. Shida]
{Graduate School of Mathematics, 
Nagoya University, 
Furocho, Chikusa-ku, Nagoya 464-8602, Japan}

\address[N. Tomita]
{Department of Mathematics, 
Graduate School of Science, Osaka University, 
Toyonaka, Osaka 560-0043, Japan}

\email[T. Kato]{kato.tomoya.s3@f.gifu-u.ac.jp}
\email[A. Miyachi]{miyachi@lab.twcu.ac.jp}
\email[N. Shida]{naoto.shida.c3@math.nagoya-u.ac.jp}
\email[N. Tomita]{tomita@math.sci.osaka-u.ac.jp}

\date{\today}

\keywords
{Bilinear Fourier multipliers, 
bilinear oscillatory integral operators}

\thanks
{This work was supported by JSPS KAKENHI, 
Grant Numbers 
23K12995 (Kato), 
20H01815 (Miyachi), 
23KJ1053 (Shida), 
and 
20K03700 (Tomita).}

\subjclass[2020]
{42B15, 
42B20}

\begin{abstract}
For bilinear Fourier multipliers that contain 
some oscillatory factors, boundedness of the operators 
between Lebesgue spaces is given including 
endpoint cases. 
Sharpness of the result is also considered. 
\end{abstract}

\maketitle

\section{Introduction}\label{intro}
Throughout this paper, the letter $n$ denotes a positive integer.

For a bounded function $\sigma = \sigma(\xi)$ on $\R^n$,
the linear Fourier multiplier operator $\sigma(D)$ is defined by
\[
\sigma(D)f(x) 
=
\frac{1}{(2\pi)^n} 
\int_{\R^n}
e^{ix \cdot \xi}
\sigma(\xi)
\widehat{f}(\xi)
\,
d\xi,
\quad
x \in \R^n,
\]
for $f \in \calS(\R^n)$, where
$\widehat{f}$ denotes the Fourier transform of $f$.
If $X$ is a function space on $\R^n$
equipped with the quasi-norm $\|\cdot\|_{X}$
and there exists $c > 0$
such that
\[
\|\sigma(D)f\|_{X} \le c \|f\|_{X}
\,\,\,\,
\text{for all}
\,\,\,\,
f \in \calS \cap X,
\]
then we say that $\sigma(D)$ is bounded on $X$.

We recall the result for the multiplier of the form
\[
e^{i|\xi|^s} \zeta(\xi) |\xi|^{m}, 
\quad
0 < s  < 1
\,\,\, 
\text{or}
\,\,\,
1 < s < \infty, 
\quad
m \in \R,
\]
where $\zeta$ is $C^\infty$ function on $\R^n$ such that
$\zeta(\xi) = 0$ for $|\xi| \le 1$
and
$\zeta(\xi) = 1$ for $|\xi| \ge 2$ (see also Notation \ref{notation}).

\begin{referthmA}
[\cite{Sjolin-MathZ, Miyachi-multiplier}]
\label{lin_bdd_Hp}
Let $m \in \R$, and let $0< s < 1$ or $1 < s < \infty$,
and let $1 \le p \le \infty$.
Then the Fourier multiplier operator 
$e^{i|D|^s} \zeta(D) |D|^{m}$
is bounded
on $H^p$ when $p < \infty$
and on $BMO$ when $p= \infty$
if and only if 
$m \le -ns|1/p -1/2|$.
\end{referthmA}

Here, $H^p$, $0 < p \le \infty$, denotes the Hardy space 
and the space $BMO$ denotes the space of bounded mean oscillation.
It is known that $H^p = L^p$ if $1 < p \le \infty$ 
and $H^1 \hookrightarrow L^1$.
For details on these function spaces, 
see, {\it e.g.,} \cite[Chapters $\III$ and $\IV$]{Steinbook}.

Next, we shall consider the bilinear case.
For a bounded function $\sigma = \sigma(\xi, \eta)$ on $\R^{2n}$,
the bilinear Fourier multiplier operator $T_{\sigma}$
is defined by
\[
T_{\sigma} (f, g)(x)
=\frac{1}{(2\pi)^{2n}}
\iint_{\R^n\times \R^n} 
e^{i x \cdot(\xi + \eta)} \, 
\sigma (\xi, \eta) \,
\widehat{f}(\xi)\, 
\widehat{g}(\eta)
\, d\xi d\eta,
\quad
x \in \R^n,
\]
for $f, g \in \calS(\R^n)$. 
For function spaces on $\R^n$, $X$, $Y$ and $Z$
equipped with the quasi-norms $\|\cdot\|_X$, $\|\cdot\|_Y$ and $\|\cdot\|_Z$,
respectively,
we say that $T_{\sigma}$ is bounded from $X \times Y$ to $Z$, 
or $T_{\sigma}$ is bounded in $X \times Y \to Z$
if there exists $C > 0$ such that
\begin{align*}
\|T_{\sigma}(f, g)\|_{Z}
\le
C
\|f\|_{X}
\|g\|_{Y}
\,\,\,
\text{for all}
\,\,\,
f \in \calS \cap X
\,\,\,
\text{and all}
\,\,\,
g \in \calS \cap Y.
\end{align*}
We define the operator norm $\|T_{\sigma}\|_{X \times Y \to Z}$
to be the smallest constant $C$ in the above inequality.

In this paper, we especially consider  the bilinear Fourier multiplier operator 
$T^{s}_{\sigma}$, $0 < s < \infty$, of the following form:
\[
T_{\sigma}^{s} (f, g)(x)
=\frac{1}{(2\pi)^{2n}}
\iint_{\R^n\times \R^n} 
e^{i x \cdot(\xi + \eta)} \, 
e^{i (|\xi|^{s} + |\eta|^{s})} \, 
\sigma (\xi, \eta) \,
\widehat{f}(\xi)\, 
\widehat{g}(\eta)
\, d\xi d\eta,
\quad
x \in \R^n,
\]
for $f, g \in \calS(\R^n)$.

In order to describe the results on the bilinear operators of this type,
we define the class $S^{m}_{1, 0}(\R^{2n})$ as follows.

\begin{defn}\label{def-Sm10} 
For $m \in \R$, 
the class $S^{m}_{1,0} (\R^{2n})$ is defined to be the set of all 
$C^{\infty}$ functions $\sigma = \sigma (\xi, \eta)$ 
on $\R^{2n}$ 
that satisfy the estimate 
\begin{equation*}
\big|
\partial^{\alpha}_{\xi} 
\partial^{\beta}_{\eta} 
\sigma (\xi, \eta)
\big| 
\le 
C_{\alpha, \beta} \big( 1+ |\xi| + |\eta| \big)^{m-|\alpha|-|\beta|}
\end{equation*}
for all multi-indices 
$\alpha, \beta \in (\N_0)^{n} = (\{ 0,1,2, \dots \})^{n}$. 
\end{defn}

For the case $s=1$, 
Grafakos--Peloso \cite{GP-JPDOA}
first gave
the boundedness results for such kind of operators,
which were developed in the series of the papers
\cite{RS-JFA, RRS-AM, RRS-TAMS}
by the authors S.~Rodr\'iguez-Lop\'ez, D. Rule, 
and W. Staubach.
Quite recently, the first, the second and the last authors of the present paper 
improve these results in \cite{KMT-wave}.
Although the present paper is inspired by \cite{KMT-wave}, 
since our subject concerns with the case $s \neq 1$,
we omit to mention the details on the results of
\cite{RS-JFA, RRS-AM, RRS-TAMS, KMT-wave}.

For the case $s \neq 1$, 
Bergfeldt--Rodr\'iguez-Lop\'ez--Rule--Staubach \cite{BRRS-TAMS} 
recently considered the bilinear operator $T_{\sigma}^s$,
and proved the following theorem.
\begin{thm} [{\cite[Theorem 1.4 and Remark 1.5]{BRRS-TAMS}}] 
\label{thm_BRRS}
Let $0 < s < 1$ or $1 < s < \infty$
and let
$1\le p, q\le \infty$ and $1/r = 1/p+1/q$.
Suppose that
$\sigma \in S^{m}_{1,0} (\R^{2n})$
with
$m = -ns \big( | {1}/{p} - {1}/{2} | 
+ | {1}/{q} - {1}/{2} | \big)$.
Then,
$T_{\sigma}^{s}$ is bounded from  $H^{p} \times H^{q}$ to $L^{r}$,
where $L^r$ should be replaced by $BMO$ when $r=\infty$. 
\end{thm}

Here we give a remark on Theorem \ref{thm_BRRS}.
The verbatim statement of \cite[Theorem 1.4 and Remark 1.5]{BRRS-TAMS} 
contains the restriction $r > \frac{n}{n+ \min\{1, s\}}$. 
However, if we carefully read the paper, 
we see that this restriction can be removed. 
For the reader's convenience, 
we shall give an independent 
proof of Theorem \ref{thm_BRRS} in Section \ref{prf_BRRS}. 

Now, the purpose of this paper is to give an 
improvement of Theorem \ref{thm_BRRS}.
To state our main result, we prepare some notations.
We divide the set $\{ 1 \le p,q \le \infty \} \subset \R^2$ 
into the following six subsets:
\begin{align*} 
\RomI 
&= \{ 2 \le p,q \le \infty \} ,
\\
\II
&= \{ 1 \le p,q \le 2 \} ,
\\
\III
&= 
\{ 1 \le p \le 2 \le q \le \infty \;\;\textrm{and}\;\; 1/p+1/q \le 1 \} ,
\\
\IV
&=
\{ 1 \le p \le 2 \le q \le \infty \;\;\textrm{and}\;\; 1/p+1/q \ge 1 \} ,
\\
\RomV
&= 
\{ 1 \le q \le 2 \le p \le \infty \;\;\textrm{and}\;\; 1/p+1/q \le 1 \},
\\
\VI
&= \{ 1 \le q \le 2 \le p \le \infty \;\;\textrm{and}\;\; 1/p+1/q \ge 1 \} ,
\end{align*}
which satisfy $\RomI \cup \II \cup \III \cup \IV \cup \RomV \cup \VI = \{ 1 \le p,q \le \infty \} $
and are assigned into the picture below.
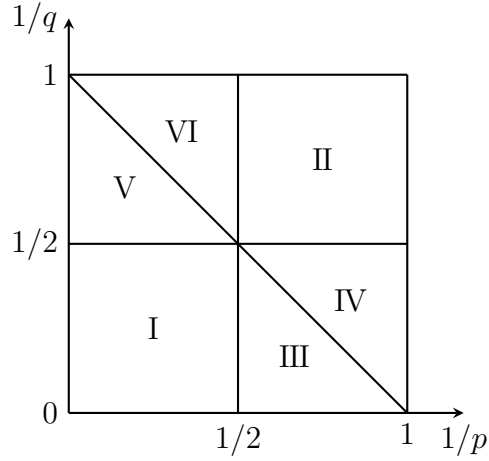
\begin{center}
\begin{tikzpicture}[scale=0.75]
   \draw [thick, -stealth](-0,0)--(7,0) node [anchor=north]{$1/p$};
   \draw [thick, -stealth](0,-0)--(0,7) node [anchor=east]{$1/q$};
   \node [anchor=east] at (0,0) {0};
   \draw [thick](3,0)--(3,6);
   \draw [thick](0,3)--(6,3);
   \draw [thick](6,0)--(6,6);
   \draw [thick](0,6)--(6,6);
   \draw [thick](0,6)--(6,0);
   \node [anchor=north] at (3,0) {1/2};
   \node [anchor=east] at (0,3) {1/2};
   \node [anchor=north] at (6,0) {1};
   \node [anchor=east] at (0,6) {1};
   \node [font=\normalsize] at (1.5,1.5) {$\RomI$};
   \node [font=\normalsize] at (4.5,4.5) {$\II$};
   \node [font=\normalsize] at (5,2) {$\IV$};
   \node [font=\normalsize] at (1,4) {$\RomV$};
   \node [font=\normalsize] at (2,5) {$\VI$};
   \node [font=\normalsize] at (4,1) {$\III$};
\end{tikzpicture} 
\end{center}
Using these sets, 
we define 
$m_s (p,q)$
by
\begin{align*} &
m_s (p,q)
= 
\begin{cases}
-ns \big( | \frac{1}{p} - \frac{1}{2} |
 + | \frac{1}{q} - \frac{1}{2} | \big)
	& \textrm{for}\;\; 
	(p,q) \in \RomI \cup \II, 
\\
-ns(1-s) \big| \frac{1}{p} - \frac{1}{2} \big|
-ns \big| \frac{1}{q} - \frac{1}{2} \big|
	& \textrm{for}\;\; 
	(p,q) \in \III \cup \VI,
\\
-ns | \frac{1}{p} - \frac{1}{2} |
-ns (1-s)| \frac{1}{q} - \frac{1}{2} | 
	& \textrm{for}\;\; 
	(p,q) \in \IV \cup \RomV, 
\\
\end{cases}
\quad
\textrm{when}\;\;
0 < s < 1, 
\end{align*}
and
\begin{align*}
m_s(p, q)
= 
\begin{cases}
-ns \big( | \frac{1}{p} - \frac{1}{2} |
 + | \frac{1}{q} - \frac{1}{2} | \big)
	& \textrm{for}\;\; 
	(p,q) \in \RomI \cup \II, 
\\
-ns \big| \frac{1}{q} - \frac{1}{2} \big|
	& \textrm{for}\;\; 
	(p,q) \in \III \cup \VI, 
\\
-ns \big| \frac{1}{p} - \frac{1}{2} \big|
	& \textrm{for}\;\; 
	(p,q) \in \IV \cup \RomV, 
\\
\end{cases}
\quad
\textrm{when}\;\;
1 < s < \infty .
\end{align*}

The  main result of this paper reads as follows.

\begin{thm} \label{thm_main}
Let 
$0 < s < 1$ or $1 < s < \infty$
and let $1\le p, q\le \infty$ and $1/r = 1/p+1/q$.
Suppose that
$\sigma \in S^{m}_{1,0} (\R^{2n})$
with
$m = m_s (p,q)$.
Then $T_{\sigma}^{s}$ is bounded 
from $H^{p} \times H^{q}$ to $L^{r}$,
where $L^r$ should be replaced by $BMO$ when $r=\infty$. 
\end{thm}

For $0 < s < 1$ or $1 < s < \infty$,
the number $m_s(p, q)$ is always bigger than or equal to 
the number $-ns(|1/p -1/2| + |1/q -1/2|)$
for all $1 \le p, q \le \infty$.
In particular, if 
$(p, q) \in \III \cup \IV \cup \RomV \cup \VI$,
then 
$m_s(p, q) > -ns(|1/p -1/2| + |1/q -1/2|)$ 
except for $p = 2$ or $q = 2$. 
In this sense, Theorem \ref{thm_main} improves Theorem \ref{thm_BRRS}.
Moreover, the number $m_{s}(p, q)$ defined above is 
optimal for some cases.
More precisely, the following theorem holds true.

\begin{thm} \label{thm_necessity}
Let $0 < s < 1$ or $1 < s < \infty$,
and let
$m \in \R$, 
$(p, q) \in \RomI \cup \II \cup \IV \cup \VI$, and $1/ r = 1/p + 1/q$.
Suppose that 
all $T_{\sigma}^{s}$ with 
$\sigma \in S^{m}_{1,0}(\R^{2n})$  
are bounded from $H^{p} \times H^{q}$ to $L^{r}$
with $L^r$ replaced by $BMO$ when $r = \infty$.
Then $m \le m_s (p,q)$. 
\end{thm}

It should be emphasized that the optimality for the case
$(p, q) \in \RomI \cup \II$ is already proved in 
\cite[Section 3.2]{BRRS-TAMS}. 
However, we shall also give the proofs of these cases,
which will be slightly different from the one in 
\cite[Section 3.2]{BRRS-TAMS}.

The rest of this paper is organized as follows.
In Section \ref{prf_BRRS}, 
we give a proof of Theorem \ref{thm_BRRS}.
In Section \ref{FourierTrans},
we consider 
the asymptotic behavior of 
the Fourier transform of 
functions including an oscillator
$e^{i|\xi|^s}$,
which will play important roles
in the proofs of Theorems \ref{thm_main}
and \ref{thm_necessity}.
In Section \ref{lemmas}, 
we prepare some lemmas 
which will be used 
in Section \ref{Bdd_H1LinftyL1}.
In Section \ref{Bdd_H1LinftyL1}, 
we prove 
the assertion 
of Theorem \ref{thm_main} 
in the end point case $(p, q) = (1, \infty)$,
which implies Theorem \ref{thm_main} 
with the aid complex interpolation.
In Section \ref{Nec_on_m}, 
we prove Theorem \ref{thm_necessity}.

We end this section by preparing some notations.

\begin{notation}\label{notation}

We denote by $\N$ and $\N_0$ 
the sets of positive integers and nonnegative integers, respectively.

The Fourier transform and the inverse Fourier transform 
on $\R^n$ are defined by 
\begin{align*}&
\widehat{f}(\xi)
=
\int_{\R^n}
e^{-i \xi \cdot x}
f(x)\, dx
\quad\textrm{and}\quad
(g)^{\vee}(x)
=
\frac{1}{(2\pi)^n} 
\int_{\R^n}
e^{i \xi \cdot x}
g(\xi)\, d\xi. 
\end{align*}

We take $\varphi, \psi \in \calS (\R^n)$ 
such that 
$\varphi = 1$ on 
$\{ |\xi| \le 1 \}$,
$\supp \varphi \subset 
\{ |\xi| \le 2 \}$,
$\supp \psi \subset 
\{ 1/2 \leq |\xi| \leq 2 \}$,
and
$\varphi+\sum_{j\in\N} \psi ( 2^{-j} \cdot ) = 1$.
In what follows, we will write
$\psi_{0} = \varphi$,
$\psi_j = \psi ( 2^{-j} \cdot )$ 
for $j \in \N$, 
and 
$\varphi_j = \varphi ( 2^{-j} \cdot )$ 
for $j \in \N_{0}$.
Then, we see that 
$\varphi_0 = \psi_0 = \varphi$ and
\begin{align*}
\sum_{j=0}^{k} \psi_j = \varphi_k,  
\quad k\in \N_0. 
\end{align*}
We define the $C^{\infty}$ function
$\zeta = 1 - \varphi$.
Then we have
$\partial^{\alpha} \zeta 
\in C^{\infty}_{0} (\R^n)$
for $|\alpha|\ge 1$,
$\zeta = \sum_{j \in \N} \psi_{j}$,
and
\begin{align*} &
\zeta =0 \;\; \text{on}\;\; \{ |\xi|\le 1 \},
\quad
\zeta =1 \;\; \text{on}\;\; \{ |\xi|\ge 2 \}.
\end{align*}

For a smooth function $\theta$ on $\R^n$ and 
for $N \in \N_{0}$, we write 
$\|\theta\|_{C^N}=
\max_{|\alpha|\le N} 
\sup_{\xi} \big| \partial_{\xi}^{\alpha} \theta (\xi) \big|$. 

Lastly, we recall 
the local Hardy space $h^1$
(for the definition of 
the local Hardy space $h^1$,
see Goldberg \cite{goldberg}).
It is known that
$H^1 \hookrightarrow h^1 \hookrightarrow L^1$.
As proved in \cite{goldberg},
all functions in $h^1$ 
can be decomposed
by so-called {\it atoms},
which satisfy that
\begin{equation} \label{atom_r>1}
\supp f \subset \{ y \in \R^n \mid |y-\bar{y}| \le r \} ,
\quad
\|f\|_{L^{\infty}}\le r^{-n} ,
\end{equation}
and, in addition,
if $r < 1$,
\begin{align} \label{atom_moment}
\int f(y)\, dy = 0 .
\end{align}
It is easily proved that, if $f$ satisfies only 
\eqref{atom_r>1} with $r \ge 1$,
then $f$ can be written as a linear combination of the atoms 
that satisfy \eqref{atom_r>1} with $r=1$
(see, {\it e.g.}, Miyachi--Tomita \cite{MT-IUMJ}).
In this paper,
a function $f$ on $\R^n$ 
is called
an {\it $h^1$-atom of first kind} 
if $f$ satisfies 
\eqref{atom_r>1}-\eqref{atom_moment} for $r < 1$,
and is called
an {\it $h^1$-atom of second kind} 
if $f$ satisfies \eqref{atom_r>1} for $r = 1$.
Atoms of both kinds are simply called {\it $h^1$-atoms}.

\end{notation}

\section{Proof of Theorem \ref{thm_BRRS}}
\label{prf_BRRS}

In this section, we shall give 
a proof of Theorem \ref{thm_BRRS}.
The ideas of the proof come from \cite[Proof of Theorem 1.3]{KMT-wave}.

For $d \in \N$ and $m \in \R$, 
the class $\dotS^{m}_{1, 0}(\R^d)$
consists of all $C^{\infty}$ functions $\sigma$ on $\R^d \setminus \{0\}$
such that 
\[
|\partial^\alpha_{\xi} \sigma(\xi)|
\le
C_{\alpha}
|\xi|^{m-|\alpha|},
\quad
\xi \in \R^d \setminus \{0\}
\]
for all multi-indices $\alpha \in (\N_0)^d$.
We use the notation $X_r$ given  by
\begin{align} \label{funcspXr}
X_r
=
\begin{cases}
L^r
&
\text{if}
\,\,\,\,
0 < r < \infty,
\\
BMO
&
\text{if}
\,\,\,\,
r = \infty.
\end{cases}
\end{align}

Now, we recall the boundedness result
for the bilinear Fourier multiplier operator $T_{\sigma}$ 
with $\sigma \in \dotS^{0}_{1, 0}(\R^{2n})$.
The following theorem is due to 
Coifman-Meyer \cite{CM-Ast, CM-AIF},
Kenig-Stein \cite{KS-MRL}, 
Grafakos-Kalton \cite{GK-CM},
and
Grafakos-Torres \cite{GT-AM}.

\begin{thm}
\label{Thm_S010}
Let $0 < p, q \le \infty$ and $1/r = 1/p + 1/q$.
If $\sigma \in \dotS^{0}_{1, 0}(\R^{2n})$, then
the bilinear Fourier multiplier operator 
$T_{\sigma}$ is bounded from
$H^p \times H^q$ to $X_r$.
\end{thm}

We will use the following two propositions,
whose proofs can be found in \cite[Section 6]{KMT-wave}.

\begin{prop}[{\cite[Proposition 2.3]{KMT-wave}}] \label{bdd_flagparapro1}
Let $m_1, m_2 \le 0$, $m=m_1+m_2$, 
$a_0 \in \dotS^m_{1, 0}(\R^{2n})$, 
$a_1 \in \dotS^{-m_1}_{1, 0}(\R^n)$,
$a_2 \in \dotS^{-m_2}_{1, 0}(\R^n)$, 
and
$\sigma(\xi, \eta) = a_0(\xi, \eta) a_1(\xi) a_2(\eta)$.
Then the bilinear Fourier multiplier operator 
$T_{\sigma}$ is bounded in 
\begin{align*}
\begin{cases}
H^{p} \times H^{q} \to L^{r},
&
0 < p, q < \infty, 
\,\,\,\,
1/r = 1/p + 1/q,
\\
BMO \times H^{q} \to L^{q},
&
0 < q < \infty,
\,\,\,\,
\text{if} 
\,\,\,\,
m_1 < 0,
\\
H^{p} \times BMO \to L^{p},
&
0 < p < \infty,
\,\,\,\,
\text{if} 
\,\,\,\,
m_2 < 0,
\\
BMO \times BMO \to BMO
&
\text{if}
\,\,\,\,
m_1, m_2 < 0.
\end{cases}
\end{align*}
\end{prop}

\begin{prop}[{\cite[Proposition 2.4]{KMT-wave}}] \label{bdd_flagparapro2}
Let $m_1 \le  0$, 
$a_0 \in \dotS^{m_1}_{1, 0}(\R^{2n})$, 
$a_1 \in \dotS^{-m_1}_{1, 0}(\R^n)$, 
and let
$\sigma(\xi, \eta) = a_0(\xi, \eta) a_1(\xi)$.
Then the bilinear Fourier multiplier operator $T_{\sigma}$
is bounded in
\begin{align*}
\begin{cases}
H^{p} \times L^\infty \to L^{p},
&
0 < p < \infty,
\\
BMO \times L^\infty \to BMO
&
\text{if}
\,\,\,\,
m_1 < 0.
\end{cases}
\end{align*}
\end{prop}

\begin{proof}[Proof of Theorem \ref{thm_BRRS}]

Let $0 < s < 1$ or $s > 1$, and let $\sigma \in S^{m}_{1, 0}(\R^{2n})$,
$m = -ns(|1/p-1/2| + |1/q-1/2|)$.
We write $m_1 = -ns|1/p -1/2|$ and $m_2 = -ns|1/q -1/2|$.
Using the functions $\zeta$ and $\varphi$ given in Notation \ref{notation},
we decompose the bilinear multiplier $\tau$
defined by 
$\tau(\xi, \eta) = e^{i|\xi|^s} e^{i|\eta|^s} \sigma(\xi, \eta)$
as
\begin{align*}
&
\tau (\xi, \eta)
=
\tau_{1}(\xi, \eta)
+
\tau_{2}(\xi, \eta)
+
\tau_{3}(\xi, \eta)
+
\tau_{4}(\xi, \eta),
\\
&
\tau_{1}(\xi, \eta)
=
e^{i|\xi|^s} \varphi(\xi)
e^{i|\eta|^s} \varphi(\eta)
\sigma (\xi, \eta),
\,\,\,\,\,\,
\tau_{2}(\xi, \eta)
=
e^{i|\xi|^s} \zeta(\xi)
e^{i|\eta|^s} \varphi(\eta)
\sigma (\xi, \eta),
\\
&
\tau_{3}(\xi, \eta)
=
e^{i|\xi|^s} \varphi(\xi)
e^{i|\eta|^s} \zeta(\eta)
\sigma(\xi, \eta),
\,\,\,\,\,\,
\tau_{4}(\xi, \eta)
=
e^{i|\xi|^s} \zeta(\xi)
e^{i|\eta|^s} \zeta(\eta)
\sigma(\xi, \eta).
\end{align*}
We show that each $T_{\tau_i}$, $i=1, 2, 3, 4$, is bounded
from $H^p \times H^q$ to $X_r$, 
$1 \le p, q \le \infty$, $1/r = 1/p + 1/q$.

We begin with the estimate of $T_{\tau_1}$.
Since $(e^{i|\xi|^s} \varphi(\xi))^{\vee} \in L^1(\R^n)$
(see \eqref{kernels} and  \eqref{Linq}), 
the Fourier multiplier operator $e^{i|D|^s} \varphi(D)$ 
is bounded on $H^p$, $1 \le p \le \infty$.
Since $m \le 0$, we have 
$\sigma \in S^{m}_{1, 0}(\R^{2n}) 
\subset S^{0}_{1, 0}(\R^{2n}) 
\subset \dotS^{0}_{1, 0}(\R^{2n})$,
and hence by Theorem \ref{Thm_S010}, 
$T_{\sigma}$ is bounded from $H^{p} \times H^{q}$ to $X_r$.
Thus, the desired boundedness of $T_{\tau_1}$ is given.

Next, we consider the estimate of $T_{\tau_2}$.
We write
\begin{align*}
\tau_2(\xi, \eta)
=
e^{i|\xi|^s}
\zeta(\xi)
|\xi|^{m_1}
\times
e^{i|\eta|^s}
\varphi(\eta)
\times
|\xi|^{-m_1}
\sigma (\xi, \eta).
\end{align*}
By Theorem A, the Fourier multiplier operator
$e^{i|D|^s} \zeta(D) |D|^{m_1}$ 
is bounded on $H^p$ if $1 \le p < \infty$,
and on $BMO$ if $p= \infty$. 
As we showed above,
$e^{i|D|^s} \varphi(D)$ 
is bounded on 
$H^q$, $1 \le q \le \infty$.
On the other hand,
since $|\xi|^{-m_1} \in \dotS^{-m_1}_{1, 0}(\R^{n})$
and $\sigma \in S^{m}_{1, 0}(\R^{2n}) \subset \dotS^{m_1}_{1, 0}(\R^{2n})$, 
from Propositions \ref{bdd_flagparapro1} and \ref{bdd_flagparapro2},
it follows that 
the bilinear Fourier multiplier operator
corresponding to 
$|\xi|^{-m_1} \sigma (\xi, \eta)$
is bounded in
$H^p \times H^q \to X_r$,
with $H^p$ replaced by $BMO$
if $p = \infty$
(notice that $m_1 < 0$ if $p = \infty$).
Thus, combining these boundedness, 
we obtain  the $H^p \times H^q \to X_r$
boundedness of $T_{\tau_2}$.

In the same way as above, we see that 
$T_{\tau_3}$ is bounded from $H^p \times H^q$ to $X_r$.

We finally prove that $T_{\tau_4}$ is bounded from 
$H^p \times H^q$ to $X_r$.
The multiplier $\tau_4$ can be written as 
\begin{align*}
\tau_4(\xi, \eta)
=
e^{i|\xi|^s}
\zeta(\xi)
| \xi |^{m_1}
\times
e^{i|\eta|^s}
\zeta(\eta)
| \eta |^{m_2}
\times
| \xi |^{-m_1}
| \eta |^{-m_2}
\sigma (\xi, \eta).
\end{align*}
Since 
$\sigma \in S^{m}_{1, 0}(\R^{2n}) \subset \dotS^{m}_{1, 0}(\R^{2n})$,
$|\xi|^{-m_1} \in \dotS^{-m_1}_{1, 0}$
and
$|\eta|^{-m_2} \in \dotS^{-m_2}_{1, 0}$,
it follows from
Proposition \ref{bdd_flagparapro1} 
that the bilinear Fourier multiplier 
$| \xi |^{-m_1} | \eta |^{-m_2} \sigma (\xi, \eta)$
gives rise to a bounded operator in 
$H^{p} \times H^{q} \to X_{r}$
with $H^p$ or $H^q$ replaced by $BMO$
if $p = \infty$ or $q = \infty$, respectively.
Here, we notice that $m_1< 0$
if $p = \infty$,
and $m_2 < 0$ if $q = \infty$, respectively.
Hence, combining this with Theorem A, 
we obtain 
the $H^{p} \times H^{q} \to X_{r}$ boundedness of $T_{\tau_4}$.
This completes the proof of Theorem \ref{thm_BRRS}.
\end{proof}

\section{Fourier transform of $e^{ i |\xi|^s } \psi (2^{-j}\xi)$}
\label{FourierTrans}

In this section, we investigate the asymptotic behavior of 
the Fourier transform of the oscillator $e^{i|\xi|^s}$
multiplied by Littlewood-Paley's dyadic decompositions.
This property is one of the keys to proving our main theorem.

\begin{prop} \label{Kj_psi}
Let $0< s < 1$ or $1 < s < \infty$.
Suppose that $\psi \in \calS(\R^n)$ satisfies
$\supp \psi \subset \{ 1/2 \le |\xi| \le 2 \}$.
Then, for any $N_{1}, N_{2}, N_{3} \ge 0$,
there exist $c=c(n,s, N_1, N_2, N_3)>0$ 
and $M=M(n, s, N_1, N_2, N_3) \in \N$
such that
\begin{equation}\label{Kjx0}
\Big| \Big( 
e^{ i |\xi|^{s} } 
\psi (2^{-j} \xi ) 
\Big)^{\vee} (x)
\Big|
\le c\,
\| \psi \|_{ C^M }
\begin{cases}
{2^{-jN_1}} ,
	& \text{if}\;\; {2^{j (1-s)} |x| < a}, \\
{2^{j(n-\frac{ns}{2})}} ,
	& \text{if}\;\; {a\le 2^{j (1-s)} |x| \le b}, \\
{2^{-jN_2}|x|^{-N_3}} ,
	& \text{if}\;\; {2^{j (1-s)} |x| >b} ,
\end{cases}
\end{equation}
for $j \in \N_{0}$,
where
$a = s 4^{ -|1-s| }$
and 
$b = s 4^{ |1-s| }$.
If in addition $\psi (\xi)\neq 0$ for $2/3 \le |\xi| \le 3/2$, then 
there exist $c^{\prime}=c^{\prime}(n, s, \psi)>0$ and 
$j_0=j_0 (n,s, \psi)\in \N$ such that  
\begin{equation} \label{Kjxradial}
\begin{split}
&
\frac{1}{c^{\prime}}
2^{j(n - \frac{ns}{2})}\le 
\Big| \Big( 
e^{ i |\xi|^{s} } 
\psi (2^{-j} \xi ) 
\Big)^{\vee} (x)
\Big|
\le 
c^{\prime}
2^{j(n - \frac{ns}{2})}
\\
&
\text{if}
\quad
a^{\prime} \le 2^{j(1-s)}|x| \le b^{\prime} 
\;\; 
\text{and}
\;\;
j>j_0, 
\end{split}
\end{equation}
where
$a^{\prime} = s (3/2)^{ -|1-s| }$
and 
$b^{\prime} = s (3/2)^{ |1-s| }$.
\end{prop}

To prove this proposition, we first observe that the 
determinant and the signature 
($=$({the number of positive eigenvalues})
$-$({the number of negative eigenvalues})) 
of the matrix 
\[
\Hess\, \big( |\xi|^{s}\big) = 
\big( \partial_{\xi_i} \partial_{\xi_j} |\xi|^{s} \big)_{1\le i,j \le n}
\]
are given by 
\begin{equation}\label{det}
\det \Hess\,  \big( |\xi|^{s}\big) 
=
s^n (s -1) |\xi|^{(s -2)n}
\end{equation}
and 
\begin{equation}\label{sign}
\sign \Hess\, \big( |\xi|^{s}\big) 
=
\begin{cases}
{n-2} & \text {if}\;\;0<s <1, \\
{n} & \text {if}\;\;s >1. \\
\end{cases}
\end{equation}
In fact this is a simple computation. 
We have 
\begin{equation*}
\partial_{\xi_i} \partial_{\xi_j} |\xi|^{s} 
=s |\xi|^{s-2} \delta_{i,j} 
+ s (s -2) |\xi|^{s -2} \frac{\xi_i}{|\xi|}\, \frac{\xi_j}{|\xi|}. 
\end{equation*}
Hence if we take an orthogonal matrix $T=(t_{i,j})$ that satisfies 
\[
\sum_{j=1}^{n} t_{i,j} \frac{\xi_j}{|\xi|}
=
\begin{cases}
{1} & \text {for}\;\;i=1, \\
{0} & \text {for}\;\;i=2, \dots , n, \\
\end{cases}
\]
then $T (\Hess ( |\xi|^{s}) ) T^{-1}$  
is equal to the diagonal matrix with the diagonal entries 
\[
s(s-1) |\xi|^{s-2},\; 
s |\xi|^{s-2}, \; \dots ,  \; s |\xi|^{s-2}. 
\]
From this we obtain \eqref{det} and \eqref{sign}.

\begin{proof}[Proof of Proposition \ref{Kj_psi}]
By a simple change of variables we can write 
\begin{equation*}
H_j (x) 
=\Big( 
e^{ i |\xi|^{s} } 
\psi (2^{-j} \xi ) 
\Big)^{\vee} (x)
=
\frac{2^{jn}}{(2\pi)^n} 
\int_{\R^n} 
e^{i 2^{js} \phi_j (x, \eta)} 
\psi (\eta)\, d\eta   
\end{equation*}
with  
\[
\phi_j (x,\eta) = 2^{j(1-s)} x\cdot \eta + |\eta|^s. 
\]

The gradient of the phase function $\phi_j (x,\eta)$ is 
given by 
\[
\grad_{\eta}\, \phi_j (x,\eta)
= 2^{j(1-s)} x + s |\eta|^{s-1} \frac{\eta}{|\eta|}. 
\]
For each $x \in \R^n \setminus \{0\}$, 
there exists a unique $\eta_0 = \eta_0 (x)\in \R^n \setminus \{0\}$
such that 
$\grad_{\eta} \, \phi_j (x,\eta)\big|_{\eta=\eta_0}=0$. 
In fact, $\eta_0$ is determined by the equations  
\begin{equation*}
2^{j(1-s)} |x| = s |\eta_0|^{s-1}, 
\quad 
-\frac{x}{|x|}=\frac{\eta_0}{|\eta_0|}. 
\end{equation*}
If $\eta_0$ is in a neighborhood of 
$\supp \psi $   
then we can use the stationary phase method to obtain 
the asymptotic behavior of $H_j(x)$. 
If $\eta_0$ is outside a neighborhood of 
$\supp \psi $ 
then 
we can deduce the rapid decay of $H_j (x)$ by integration by parts. 
To be precise, we divide the argument 
into several cases.

We first consider the case  $0<s<1$.

{Case $\RomI$:} $0<s<1$ and $2^{j(1-s)}|x|<s 4^{s-1}$. 
Then $|\eta_0|>4$.
In this case, for   
$\eta \in \supp \psi \subset \{{1}/{2}\le |\eta|\le 2\}$, 
we have 
\begin{align*}
&\big| 
\grad_{\eta}  \phi_j (x, \eta) \big|
=
\left| 
2^{j(1-s)} x  + s |\eta|^{s-1} \frac{\eta}{|\eta|} 
\right|
\ge 
- 2^{j(1-s)} |x| + s 2 ^{s-1}
> s (2^{s-1} - 4^{s-1})  
\end{align*}
and 
\begin{equation}\label{phiCN}
\big| \partial_{\eta}^{\alpha} 
\phi_j (x, \eta) \big|
= 
\big| \partial_{\eta}^{\alpha} 
|\eta|^{s} \big|
\le c(n, s, \alpha) 
\quad \text{for} \quad 
|\alpha|\ge 2. 
\end{equation}
Thus integration by parts gives 
\begin{equation*}
\big| H_j (x) \big| 
\le 
c(n,s,N)
\|\psi\|_{C^N}
2^{jn} (2^{js})^{-N}
\end{equation*}
for each $N\in \N$. 
Since 
$N$ can be taken arbitrarily large, 
the desired estimate 
of $H_j (x)$ in this case follows.

{Case $\II$:} 
$0<s<1$ and $2^{j(1-s)}|x|> s 4^{-(s-1)}$. 
In this case, 
$|\eta_0|< {1}/{4}$, 
and, for 
$\eta \in \supp \psi \subset \{{1}/{2}\le |\eta|\le 2\}$, 
we have 
\begin{align*}
&\big| 
\grad_{\eta}  \phi_j (x, \eta) \big|
=
\left| 
2^{j(1-s)} x  + s |\eta|^{s-1} \frac{\eta}{|\eta|} 
\right|
\ge 
2^{j(1-s)} |x| - s 2^{-(s-1)}
> 2^{j(1-s)} |x| (1 - 2^{s-1})  
\end{align*}
and we also have \eqref{phiCN}. 
Thus integration by parts gives 
\begin{equation*}
\big| H_j (x) \big| 
\le 
c(n,s,N)
\|\psi\|_{C^N}
2^{jn} (2^{j} |x|)^{-N}
\end{equation*}
for each $N\in \N$. 
Since 
$N$ can be taken arbitrarily large, 
the desired estimate 
of $H_j (x)$ in this case follows.

{Case $\III$:} 
$0<s<1$ and 
$s4^{s-1}\le 2^{j(1-s)}|x|\le s4^{-(s-1)}$. 
In this case, ${1}/{4}\le |\eta_0|\le 4$. 
By \eqref{det} and \eqref{sign}, 
we have
\begin{equation}\label{detHess}
\det\, \Hess_{\eta} \big( \phi_j (x,\eta)\big) 
=
s^n (s -1) |\eta|^{(s -2)n}<0
\end{equation}
and 
\begin{equation*}
\sign \Hess_{\eta} \big( \phi_j (x,\eta)\big) 
=
{n-2}.  
\end{equation*}
Also for each multi-index $\alpha$ there exists  
$c(n, s, \alpha)$ such that 
\begin{equation}\label{phijCalpha}
\big| \partial_{\eta}^{\alpha} \phi_j (x,\eta) 
\big| 
= 
\big| \partial_{\eta}^{\alpha} \big( 2^{j(1-s)} 
x \cdot \eta  + |\eta|^s\big)  
\big| 
\le c(n,s,\alpha) 
\quad 
\text{for}
\quad 
\frac{1}{10}< |\eta| < 10. 
\end{equation}
Notice that the constant  $c(n,s,\alpha)$ can be 
taken independent of $j$ and $x$ so long as they are in 
the range of Case $\III$. 
Thus by using 
the stationary phase method 
(see, for example, \cite[Chapter VIII, Section 2.3]{Steinbook}), 
we obtain 
\begin{equation}\label{asymptotic}
\begin{split}
H_j (x) 
=&
(2\pi)^{-\frac{n}{2}}
\exp 
\big(i 
|x|^{\frac{s}{s-1}} 
s^{\frac{-s}{s-1}} 
(1-s) 
\big) 
\big( s^n (1-s) |\eta_0|^{(s-2)n}
\big)^{-\frac{1}{2}}
e^{\frac{\pi i }{4}(n-2)} 
\\
&
\times \psi (\eta_0) 
2^{j(n-\frac{ns}{2})} 
+ O \big( 2^{j(n-\frac{ns}{2}-s)}  \big). 
\end{split}
\end{equation}
Here notice that the oscillating factor $\exp (\cdots)$ comes from 
\[
2^{js} \phi_j (x, \eta_0)
=
|x|^{\frac{s}{s-1}} 
s^{\frac{-s}{s-1}} 
(1-s). 
\]
Also notice that, by virtue of \eqref{detHess} 
and \eqref{phijCalpha}, 
the $O$-estimate in \eqref{asymptotic} holds uniformly 
for $(j,x)$ in the range of Case $\III$ and for 
$\psi$ satisfying 
$\supp \psi \subset\{1/2 \le |\xi| \le 2\}$ and 
$\|\psi\|_{C^M}\le 1$ with a sufficiently large $M=M(n)$. 
From \eqref{asymptotic}
the estimate 
of $H_j (x)$ in \eqref{Kjx0} for Case $\III$ follows.

The estimate \eqref{Kjxradial} also follows from 
\eqref{asymptotic} 
since $2/3 \le |\eta_0|\le 3/2$ if 
$
s (3/2)^{s-1 } 
\le 2^{j(1-s)}|x| 
\le 
s (2/3)^{s-1}
$.

Next we consider the case $s>1$. 
Since the argument needs only slight  modification of the case $0<s<1$, 
we shall only indicate necessary modifications. 

{Case $\RomI^{\prime}$:} $s>1$ and $2^{j(1-s)}|x|<s 4^{-(s-1)}$. 
In this case, 
$|\eta_0|<{1}/{4}$ and 
\begin{equation*}
\big| 
\grad_{\eta}  \phi_j (x, \eta) \big|
> s (2^{-(s-1)} - 4^{-(s-1)}). 
\end{equation*} 
for $\eta \in \supp \psi $. 
Integration by parts 
yields 
the desired estimate. 

{Case $\II^{\prime}$:} $s>1$ and $2^{j(1-s)}|x|> s 4^{s-1}$. 
In this case, 
$|\eta_0|> 4$ and 
\begin{equation*}
\big| 
\grad_{\eta}  \phi_j (x, \eta) \big|
> 
2^{j(1-s)} |x| (1 - 2^{-(s-1)})
\end{equation*}
for $\eta \in \supp \psi$.
Integration by parts yields 
the desired estimate.

{Case $\III^{\prime}$:} $s>1$ and $s 4 ^{-(s-1)} \le 
2^{j(1-s)}|x|\le s 4^{s-1}$. 
In this case, ${1}/{4}\le |\eta_0|\le 4$. 
By \eqref{det} and \eqref{sign}, 
we have
\begin{equation*}
\det\, \Hess_{\eta} \big( \phi_j (x,\eta)\big) 
=
s^n (s -1) |\eta|^{(s -2)n}> 0
\end{equation*}
and 
\begin{equation*}
\sign \Hess_{\eta} \big( \phi_j (x,\eta)\big) 
=n.  
\end{equation*}
The estimate \eqref{phijCalpha} also holds. 
By the stationary phase method, we obtain  
\begin{align*}
H_j (x) 
=&
(2\pi)^{-\frac{n}{2}}
\exp 
\big(i 
|x|^{\frac{s}{s-1}} 
s^{\frac{-s}{s-1}} 
(1-s) 
\big) 
\big( s^n (s-1) |\eta_0|^{(s-2)n}
\big)^{-\frac{1}{2}}
e^{\frac{\pi i }{4}n} 
\\
&
\times \psi (\eta_0) 
2^{j(n-\frac{ns}{2})} 
+ O \big( 2^{j(n-\frac{ns}{2}-s)}  \big), 
\end{align*}
from which the desired 
estimates follow. 
This completes the proof of Proposition \ref{Kj_psi}.  
\end{proof}

\begin{cor} \label{Kj_phi}
Suppose that $\theta \in \calS(\R^n)$ satisfies
$\supp \theta \subset \{ |\xi| \le 2 \}$
and the function $\zeta$ is as in Notation \ref{notation}.
Then the following hold.
\begin{enumerate}
\item \label{Kj_phi_s<1}
Let $0< s < 1$ and $N\ge 0$.
Then,
there exist $c>0$ and $M \in \N$
such that
\begin{equation*} 
\Big| \Big( 
e^{ i |\xi|^{s} } 
\zeta(\xi)
\theta (2^{-j} \xi ) 
\Big)^{\vee} (x)
\Big|
\le c\,
\| \theta \|_{ C^M }
\begin{cases}
|x|^{- \frac{n}{2} - \frac{n}{2(1-s)} } ,
	& \text{if}\;\; |x| \le 1, \\
|x|^{ -N } ,
	& \text{if}\;\; |x| > 1 ,
\end{cases}
\end{equation*}
for all $j \in \N_{0}$.

\item \label{Kj_phi_s>1}
Let $1 < s < \infty$and $N\ge 0$.
Then,
there exist $c>0$ and $M \in \N$
such that
\begin{equation*} 
\Big| \Big( 
e^{ i |\xi|^{s} } 
\zeta(\xi)
\theta (2^{-j} \xi ) 
\Big)^{\vee} (x)
\Big|
\le c\,
\| \theta \|_{ C^M }
\begin{cases}
\big( 1+|x| \big)^{- \frac{n}{2} + \frac{n}{2(s-1)} } ,
	& \text{if}\;\; |x| \le s 8^{s-1} \, 2^{j(s-1)} , \\
|x|^{ -N } ,
	& \text{if}\;\; |x| > s 8^{s-1} \, 2^{j(s-1)} ,
\end{cases}
\end{equation*}
for all $j \in \N_{0}$.
\end{enumerate}
\end{cor}

\begin{proof}
We first put
$K_{j} = 
( e^{ i |\xi|^{s} } 
\zeta(\xi)
\theta (2^{-j} \xi ) 
)^{\vee}$
and
decompose $K_{j}$
as
\begin{align} \label{decomposekernel}
K_{j} (x)
= 
\sum_{k=1}^{j+1} 
K_{k,j} (x)
\quad\textrm{with}\quad
K_{k,j} (x)
=
\Big( 
e^{ i |\xi|^{s} } \, 
\psi (2^{-k} \xi) \theta (2^{-j} \xi ) 
\Big)^{\vee} (x) .
\end{align}
Here,
we notice from 
\eqref{Kjx0} of Proposition \ref{Kj_psi}
that
\begin{equation} \label{Kkj}
\big| K_{k,j} (x) \big|
\lesssim
\| \psi (\cdot) \,
\theta (2^{k-j} \cdot) \|_{ C^M }
\begin{cases}
{2^{-kN_1}} 
	& \text{on}\;\; 
	\Omega_{k}^{1} := \{ 2^{k(1-s)} |x| \le a \}, 
\\
{2^{k(n-\frac{ns}{2})}} 
	& \text{on}\;\; 
	\Omega_{k}^{2} := \{ a< 2^{k(1-s)} |x| \le b \}, 
\\
{2^{-kN_2}|x|^{-N_3}} 
	& \text{on}\;\; 
	\Omega_{k}^{3} := \{ 2^{k(1-s)} |x| >b \} ,
\end{cases}
\end{equation}
where
$a = s 4^{ -|1-s| }$
and 
$b = s 4^{ |1-s| }$.
For $1 \le k \le j+1$,
$\| \psi (\cdot) \, \theta (2^{k-j} \cdot) \|_{ C^M } \lesssim \| \theta \|_{ C^M }$.
Hence 
$K_{j}$ is estimated as
\begin{align*} 
\big| K_{j}(x) \big|
\lesssim 
\| \theta \|_{ C^M }
\bigg(
\sum_{k=1}^{j+1}
{2^{-kN_1}} \ichi_{ \Omega_{k}^{1} }(x)
+
\sum_{k=1}^{j+1}
{2^{k(n-\frac{ns}{2})}} \ichi_{ \Omega_{k}^{2} }(x)
+
\sum_{k=1}^{j+1}
{2^{-kN_2}|x|^{-N_3}} \ichi_{ \Omega_{k}^{3} }(x)
\bigg) .
\end{align*}
Hence,
in the following argument,
we shall estimate 
the above three sums.

\eqref{Kj_phi_s<1}
Let $0<s<1$ and write
$L = \frac{n}{2} + \frac{n}{2(1-s)}$.
To prove the estimate mentioned in (1), we first prove that 
\[
|x|^{L} \big| K_{j}(x) \big|
\lesssim \| \theta \|_{ C^M },
\quad
|x| \le 1.
\]
We assume $|x| \le 1$.
For the sum with $\ichi_{\Omega_{k}^{1}}(x)$,
we have for $N_1 > 0$,
\begin{align*}
|x|^{L}
\sum_{k=1}^{j+1}
{2^{-kN_1}} \ichi_{ \Omega_{k}^{1} }(x)
&
=
\sum_{k=1}^{j+1}
\big( 2^{k (1-s)} |x| \big)^{L} \,
2^{-kL (1-s)} \,
{2^{-kN_1}} 
\ichi_{ \Omega_{k}^{1} }(x)
\lesssim
1.
\end{align*}
For the sum with $\ichi_{\Omega_{k}^{2}}(x)$,
since
$-L(1-s) + (n-\frac{ns}{2}) = 0$
and
overlaps of $\Omega_{k}^{2}$ are finite,
\begin{align*}
|x|^{L}
\sum_{k=1}^{j+1}
{2^{k(n-\frac{ns}{2})}} \ichi_{ \Omega_{k}^{2} } (x)
&
=
\sum_{k=1}^{j+1}
\big( 2^{k (1-s)} |x| \big)^{L} \,
2^{-kL (1-s)} \,
{2^{k(n-\frac{ns}{2})}} 
\ichi_{ \Omega_{k}^{2} } (x)
\approx
\sum_{k=1}^{j+1}
\ichi_{ \Omega_{k}^{2} } (x)
\lesssim
1 .
\end{align*}
For the sum with $\ichi_{\Omega_{k}^{3}}(x)$,
we have
by choosing 
$N_3 > L$ and $N_2 > (1-s)(N_3-L)$,
\begin{align*}
|x|^{L}
\sum_{k=1}^{j+1}
{2^{-kN_2}
|x|^{-N_3}} 
\ichi_{ \Omega_{k}^{3} } (x)
&
=
\sum_{k=1}^{j+1}
\big( 2^{k (1-s)} |x| \big)^{L-N_3} \,
2^{k(N_3-L) (1-s)} \,{2^{-kN_2} }
\ichi_{ \Omega_{k}^{3} } (x)
\lesssim
1 .
\end{align*}
Combining the above inequalities,
we obtain the assertion \eqref{Kj_phi_s<1}
for the case $|x| \le 1$.

We next prove that 
\[
|x|^N \big|K_j(x) \big|
\lesssim
\|\theta\|_{C^{M}},
\quad
|x| > 1,
\]
which can be shown by a similar way.
In this case,
the sum with respect to $\ichi_{ \Omega_{k}^{1} }$ vanishes.
Replacing $L$ by $N$ in the above, 
and taking $N_2, N_3>0$ satisfying that
$N_3 > N$ and $N_2 > (1-s)(N_3-N)$,
we have the desired estimate for the sum 
with respect to $\ichi_{ \Omega_{k}^{3} }$.
For the sum with respect to $\ichi_{ \Omega_{k}^{2} }$,
since $|x| > 1$ gives $2^{k(1-s)} < b$ on $\Omega_{k}^{2}$,
the cardinality of $k$ is finite.
Furthermore, since $1 < |x| \le 2^{-(1-s)}b$ if 
$x \in \Omega_{k}^{2} \cap \{|x| > 1\}$, $k \ge 1$,
it follows that 
$|x|^N \approx 1$ on $\Omega_{k}^{2} \cap \{|x| > 1\}$.
Thus, we obtain the assertion \eqref{Kj_phi_s<1}
for $|x| > 1$.

\eqref{Kj_phi_s>1}
In the case $1<s<\infty$,
we first observe that,
if $|x| \le s 2^{-(s-1)} =a 2^{s-1}$ 
or 
$|x| > s 8^{s-1} \, 2^{j(s-1)} = b 2^{(j+1)(s-1)}$,
then 
$|x| \le a2^{k(s-1)}$
or
$|x| > b2^{k(s-1)}$
holds for all $1 \le k \le j+1$,
that is,
$x \in \Omega_{k}^{1}$ or 
$x \in \Omega_{k}^{3}$.
By \eqref{decomposekernel} and \eqref{Kkj},
this implies that for any $N_1 > 0$
\begin{align*}
\big| K_{j}(x) \big|
\lesssim 
\| \theta \|_{ C^M }
\sum_{k=1}^{j+1}
{2^{-kN_1}} 
\lesssim 
\| \theta \|_{ C^M },
\quad
|x| \le s 2^{-(s-1)} ,
\end{align*}
and, for any $N_2 >0 $ and $N_3 \ge 0$,
\begin{align*}
\big| K_{j}(x) \big|
\lesssim 
\| \theta \|_{ C^M }
\sum_{k=1}^{j+1}
{2^{-kN_2}|x|^{-N_3}} 
\lesssim 
\| \theta \|_{ C^M }
|x|^{-N_3} ,
\quad
|x| > s 8^{s-1} \, 2^{j(s-1)}.
\end{align*}
Hence,
to obtain the desired result,
it suffices to prove that
\begin{equation*}
|x|^{L}
\big| K_{j}(x) \big|
\lesssim
\| \theta \|_{ C^M } 
\quad\text{on}\quad 
\Omega_{j} := 
\{ s 2^{-(s-1)} < |x| \le s 8^{s-1} \, 2^{j(s-1)} \} ,
\end{equation*}
where, we wrote 
$L = \frac{n}{2} - \frac{n}{2(s-1)}$.
Here, we note that 
$L \le 0$ for $1 < s \le 2$ and 
$L \ge 0$ for $2 \le  s < \infty$
and write $L_+ = \max\{ 0,L \}$.
Assume $x \in \Omega_{j}$.
For the sum with  $\ichi_{\Omega_{k}^{1}}(x)$, 
we have for $N_1>L_+ (s-1)$
\begin{align*}
|x|^{L}
\sum_{k=1}^{j+1}
{2^{-kN_1}} \ichi_{ \Omega_{k}^{1} }(x)
&
\lesssim
\sum_{k=1}^{j+1}
\big( 2^{k (1-s)} |x| \big)^{L_+} \,
2^{-kL_+ (1-s)} \,
{2^{-kN_1}} 
\ichi_{ \Omega_{k}^{1} }(x)
\lesssim
1.
\end{align*}
For the sum with $\ichi_{\Omega_{k}^{2}}(x)$, 
since
$-L(1-s) + (n-\frac{ns}{2}) = 0$
and
overlaps of $\Omega_{k}^{2}$ are finite,
\begin{align*}
|x|^{L}
\sum_{k=1}^{j+1}
{2^{k(n-\frac{ns}{2})}} \ichi_{ \Omega_{k}^{2} }(x)
&
=
\sum_{k=1}^{j+1}
\big( 2^{k (1-s)} |x| \big)^{L} \,
2^{-kL (1-s)} \,
{2^{k(n-\frac{ns}{2})}} 
\ichi_{ \Omega_{k}^{2} }(x)
\approx
\sum_{k=1}^{j+1}
\ichi_{ \Omega_{k}^{2} }(x)
\lesssim
1 .
\end{align*}
For the sum with $\ichi_{\Omega_{k}^{3}}(x)$,
we have
by choosing 
$N_2 > 0$ and 
$N_3 > L_+$
\begin{align*}
|x|^{L}
\sum_{k=1}^{j+1}
{2^{-kN_2}
|x|^{-N_3}} 
\ichi_{ \Omega_{k}^{3} }(x)
&
\lesssim
\sum_{k=1}^{j+1}
\big( 2^{k (1-s)} |x| \big)^{L_+-N_3} \,
2^{k(N_3-L_+) (1-s)} \,{2^{-kN_2} }
\ichi_{ \Omega_{k}^{3} }(x)
\lesssim
1 .
\end{align*}
Therefore we complete the proof of 
the assertion \eqref{Kj_phi_s>1}.
\end{proof}

\section{Lemmas}
\label{lemmas}

In this section, 
we prepare some lemmas
for our main theorems.
Let $\theta \in \calS (\R^n)$ satisfy
$\supp \theta \subset 
\{|\xi|\le 2 \} $
and $\zeta$ be as in Notation \ref{notation}.
Then, we define,
for $j \in \N$,
\begin{align*}
S_{j} f (x) 
&
=
\Big( 
e^{ i |\xi|^{s} } \, 
\zeta(\xi) \theta (2^{-j} \xi ) \,
\widehat{f} (\xi) 
\Big)^{\vee} (x) ,
\\
T f (x)
&
=
\Big( 
e^{ i |\xi|^{s} } \, 
\theta (\xi)
\widehat{f} (\xi) 
\Big)^{\vee} (x) ,
\end{align*}
which can be
represented as follows:
\begin{align} \label{kernels}
\begin{split}
S_{j} f (x) 
= 
K_{j} \ast f (x)
&
\quad\text{with}\quad
K_{j} (x)
=
\Big( 
e^{ i |\xi|^{s} } \, 
\zeta(\xi) \theta (2^{-j} \xi ) 
\Big)^{\vee} (x) ,
\\
T f (x)
= 
L \ast f (x) 
&
\quad\text{with}\quad
L (x)
=
\Big( 
e^{ i |\xi|^{s} } \, 
\theta (\xi) 
\Big)^{\vee} (x) .
\end{split}
\end{align}
Notice that
the kernel $K_j$
already appeared in
Corollary \ref{Kj_phi}.

In the succeeding subsections,
we will give several inequalities 
for the operators $S_j$ and $T$.
Some of the inequalities are concerned with 
$h^1$-atoms. 
Recall that, in our definition of $h^1$-atom, 
the radius $r$ of the supporting ball of an $h^1$-atom 
satisfies $r\le 1$ (see Notation \ref{notation}).

\subsection{Inequalities for $s \neq 1$}

In this subsection, 
we show some inequalities 
which will be used
for proving the boundedness
in both cases $s<1$ and $s>1$.

\begin{lem} \label{Sj_Lp}
Let $0 < s < 1$ or $1 < s < \infty$
and let $1 \le p \le \infty$.
Then, there exist 
$c>0$ and $M \in \N$ 
such that 
\[
\|S_j f\|_{L^p(\R^n)}
\le c\,
(2^j)^{ sn |\frac{1}{p} - \frac{1}{2}| }
\|\theta\|_{C^M}
\|f\|_{L^p(\R^n)} 
\]
for all $j \in \N_{0}$.
\end{lem}

\begin{proof}
This lemma follows from 
the trivial $L^2$-boundedness, 
the $L^1$-boundedness and $L^\infty$-boundedness 
with the aid of complex interpolation.
The $L^1$ and $L^\infty$-boundedness 
follow from the kernel estimate below:
\begin{equation} \label{estimateKjL1} 
\big\|
K_{j}
\big\|_{ L^1 (\R^n) }
\lesssim 
\| \theta \|_{ C^{M} }
(2^{j})^{ \frac{sn}{2} } 
\end{equation}
for some constant $M \in \N$.
This inequality is derived from
the following fact:
for $N \ge 0$ with
\begin{align} \label{N_no_condition}
\begin{cases}
0 \le N < \frac{n}{2(1-s)}, &
\textrm{if}\;\;
0 < s < 1,
\\
0 \le N < \infty, &
\textrm{if}\;\;
1 < s < \infty ,
\end{cases}
\end{align}
the kernel $K_j$ of $S_{j}$
satisfies that
\begin{align} \label{K_j_no_1+weighedL^2}
\Big\|
\big( 1+2^{j(1-s)} |x| \big)^{N} \, K_{j}
\Big\|_{ L^2 (\R^n) }
\lesssim 
\| \theta \|_{ C^{M} }
(2^{j})^{ \frac{n}{2} } .
\end{align}
In fact,
choosing $N\ge0$ that satisfies
\begin{align*}
\begin{cases}
\frac{n}{2} <N < \frac{n}{2(1-s)}, &
\textrm{if}\;\;
0 < s < 1,
\\
\frac{n}{2} < N < \infty, &
\textrm{if}\;\;
1 < s < \infty,
\end{cases}
\end{align*}
(notice that $\frac{n}{2} < \frac{n}{2(1-s)}$ if $0< s < 1$), 
and using Cauchy--Schwarz inequality and 
\eqref{K_j_no_1+weighedL^2}, we 
obtain 
\begin{align*}
\big\| 
K_{j}
\big\|_{ L^1 }
&
\le
\Big\| 
\big(
1+ 2^{j(1-s)}|x|
\big)^{-N}
\Big\|_{ L^2  }
\Big\|
\big(
1 + 2^{j(1-s)}|x|
\big)^N
K_{j}
\Big\|_{ L^2 }
\\
&
\lesssim
\|\theta\|_{C^M}
( 2^{j} )^{-\frac{(1-s)n}{2}}
( 2^{j} )^{\frac{n}{2}}
=
\|\theta\|_{C^M}
( 2^{j} )^{\frac{sn}{2}}.
\end{align*}
Although both \eqref{estimateKjL1} and 
\eqref{K_j_no_1+weighedL^2} can be shown by the use of 
Corollary \ref{Kj_phi}, 
here we shall give an elementary proof of \eqref{K_j_no_1+weighedL^2}, 
which may be of independent interest. 
We will also use the inequality \eqref{K_j_no_1+weighedL^2} 
in the proof of the next lemma. 

Hence, we move on proving that
\eqref{K_j_no_1+weighedL^2} holds
for $N \ge 0$ satisfying \eqref{N_no_condition}.
To this end,
it is sufficient to show that
\begin{align} \label{K_j_no_weighedL^2}
\big\|
\big( 2^{j(1-s)} |x| \big)^{N} \, K_{j}
\big\|_{ L^2 (\R^n) }
\lesssim 
\| \theta \|_{ C^{M} }
(2^{j})^{ \frac{n}{2} } ,
\quad \textrm{
if $N \ge 0$ 
satisfies \eqref{N_no_condition}.
}\;\;
\end{align}

The case $N=0$ obviously follows from
Plancherel's theorem,
and thus, we shall assume that $N>0$.
We recall the decomposition 
\eqref{decomposekernel}:
\begin{align*} 
K_{j} (x)
= 
\sum_{k=1}^{j+1} 
K_{k,j} (x)
\quad\textrm{with}\quad
K_{k,j} (x)
=
\Big( 
e^{ i |\xi|^{s} } \, 
\psi (2^{-k} \xi) \theta (2^{-j} \xi ) 
\Big)^{\vee} (x) .
\end{align*}
A simple calculation gives that
for $\alpha \in (\N_{0})^{n}$ and $1 \le k \le j+1$
\[
\Big|
\partial_{\xi}^{\alpha}
\Big(
e^{i|\xi|^s} \,
\psi (2^{-k} \xi )
\theta (2^{-j}\xi) 
\Big)
\Big|
\lesssim
\| \theta \|_{ C^{|\alpha|} }
(2^{k})^{(s-1)|\alpha|}
\ichi_{ \{ 2^{k-1} \le |\xi| \le 2^{k+1} \} } ,
\]
and thus,
by Plancherel's theorem,
\[
\big\| x^{\alpha} K_{k,j} (x) \big\|_{L^2}
\lesssim 
\| \theta \|_{ C^{|\alpha|} }
(2^{k})^{(s-1)|\alpha|} \,
(2^{k})^{ \frac{n}{2} } .
\]
Here, take $0 < t < 1$ and $\gamma \in \N$
satisfying $N = t\gamma > 0$.
Then,
by H\"older's inequality
\begin{align*}&
\big\| |x|^{N} K_{k,j} \big\|_{L^2}
= 
\Big\| 
\big( |x|^{\gamma} |K_{k,j}| \big)^{t} \,
\big| K_{k,j} \big|^{1-t} 
\Big\|_{L^2}
\\
&
\lesssim
\sum_{ |\alpha|=\gamma }
\big\| x^{\alpha} K_{k,j} \big\|_{L^2}^{t}
\big\| K_{k,j}\big\|_{L^2}^{1-t}
\lesssim
\| \theta \|_{ C^{\gamma} }
(2^{k})^{(s-1)N} \,
(2^{k})^{ \frac{n}{2} } .
\end{align*}
Therefore, 
since the condition \eqref{N_no_condition} especially means
$\frac{n}{2} + (s-1)N > 0$ in the case $0<s<1$,
\begin{align*}
\big\| |x|^{N} K_{j} \big\|_{L^2}
&
\le 
\sum_{k=1}^{j+1} 
\big\| |x|^{N} K_{k,j} \big\|_{L^2}
\lesssim
\| \theta \|_{ C^{\gamma} }
\sum_{k=1}^{j+1} 
(2^{k})^{(s-1)N} \,
(2^{k})^{ \frac{n}{2} } 
\\
&
\approx 
\| \theta \|_{ C^{\gamma} }
(2^{j})^{(s-1)N} \,
(2^{j})^{ \frac{n}{2} } .
\end{align*}
This implies \eqref{K_j_no_weighedL^2},
and thus the proof is completed.
\end{proof}

\begin{lem} \label{Sj_L2}
Let $0 < s < 1$ or $1 < s < \infty$
and let $0 \le t \le 1$.
Suppose $f$ is an $h^1$-atom 
supported on a ball of radius $r$
centered at the origin in $\R^n$. 
Then the following hold.
\begin{enumerate}

\item \label{Sj_L2_Rn}
There exists $c>0$ depending only on $n$ 
such that
\[
\big\| S_{j} f\big\|_{ L^2 (\R^n) }
\le c \,
\| \theta \|_{ C^{0} } \,
( 2^{j} )^{ \frac{n}{2} } \,
\min \big\{ (2^{j} r)^{t}, ( 2^{j} r )^{ -\frac{nt}{2} } \big\} 
\]
for all $j \in \N_{0}$. 

\item \label{Sj_L2_torus}
If $A \ge 2r$ and if $N$ satisfies that
\begin{align*}
\begin{cases}
0 \le N < \frac{n}{2(1-s)}, 
	&\textrm{if}\;\; 0 < s < 1,
\\
0 \le N < \infty, 
	& \textrm{if}\;\; 1 < s < \infty ,
\end{cases}
\end{align*}
then
\[
\big\| S_{j} f(x) \big\|_{ L^2 (A \le |x| \le 2A) }
\le c\,
\| \theta \|_{ C^{M} }
( 2^{j} )^{ \frac{n}{2} } 
\big( 2^{j(1-s)} A \big)^{ -N(1-t) }
\min \big\{ (2^{j} r)^{t}, ( 2^{j} r )^{ -\frac{nt}{2} } \big\} 
\]
for all $j \in \N_{0}$,
where
the constants $c>0$ and $M\in \N$ depend only on 
$n$, $s$, and $N$. 
\end{enumerate}
\end{lem}

\begin{proof}

\eqref{Sj_L2_Rn}
We first observe that,
by Plancherel's theorem,
\begin{align} \label{L2_Rn_t=0}
\| S_{j} f \|_{L^2}
=
\| K_{j} \ast f \|_{ L^2 }
\le
\| K_{j} \|_{ L^2 }
\| f \|_{ L^1 }
\lesssim
( 2^{j} )^{ \frac{n}{2} }
\| \theta \|_{ C^{0} }.
\end{align}
We next show that
\begin{align} \label{L2_Rn_t=1}
\| S_{j} f \|_{L^2}
\lesssim
( 2^{j} )^{ \frac{n}{2} } \,
\min \big\{ 2^{j} r, ( 2^{j} r )^{ -\frac{n}{2} } \big\} \,
\| \theta \|_{ C^{0} }
\end{align}
holds for all $h^1$-atoms $f$.
If $f$ is an $h^1$-atom of second kind 
({\it{i.e.}}, $r=1$),
then,
by Plancherel's theorem,
$\| S_{j} f \|_{ L^2 } \lesssim \| \theta \|_{ C^{0} }$,
which is identical with \eqref{L2_Rn_t=1} for $r = 1$.
We shall next consider the case that
$f$ is an $h^1$-atom of first kind ({\it{i.e.}}, $r<1$).
By Plancherel's theorem,
\[
\| S_{j} f \|_{L^2}
\le
\| \theta \|_{ C^{0} } 
\| f\|_{ L^2 (\R^n) }
\lesssim
\| \theta \|_{ C^{0} } \,
r^{ -\frac{n}{2} } 
= 
\| \theta \|_{ C^{0} } \,
( 2^{j} )^{ \frac{n}{2} } \,
( 2^{j} r)^{ -\frac{n}{2} } .
\]
Moreover,
since $f$ is an $h^1$-atom of first kind
and is supported on a ball centered at the origin,
Taylor's theorem 
with the moment condition 
$\int f = 0$
yields that
\begin{align*}
\big\| S_{j} f \big\|_{L^2}
&
=
\Big\|
\sum_{ |\alpha|=1 }
\int_{ \substack{ |y| \le r \\ 0 < t < 1 } } 
\Big( 
e^{ i |\xi|^{s} } \, 
\xi^{\alpha} \,
\zeta(\xi) \theta (2^{-j} \xi ) 
\Big)^{\vee} (x-ty) \,
y^{\alpha} \, f(y)
\, dydt 
\Big\|_{L^2_{x}}
\\
&
\le
\sum_{ |\alpha|=1 }
\Big\|
\Big( 
e^{ i |\xi|^{s} } \, 
\xi^{\alpha} \,
\zeta(\xi) \theta (2^{-j} \xi ) 
\Big)^{\vee} (x) 
\Big\|_{ L^2_{x} }
\int_{ |y| \le r } \,
|y| \, | f(y) |
\, dy
\\
&
\lesssim
\| \theta \|_{ C^{0} } \,
( 2^{j} )^{ \frac{n}{2} } \,
( 2^{j} r ) ,
\end{align*}
where,
in the last inequality,
we used Plancherel's theorem.
These two estimates
imply \eqref{L2_Rn_t=1}.

Finally,
interpolating \eqref{L2_Rn_t=0} and \eqref{L2_Rn_t=1},
we have for $0 \le t \le 1$
\begin{align*}
&
\| S_{j} f \|_{ L^2 }
=
\Big(
\| S_{j} f \|_{ L^2 }
\Big)^{1-t}
\Big(
\| S_{j} f \|_{ L^2 }
\Big)^{t}
\\
&
\lesssim
\Big(
( 2^{j} )^{ \frac{n}{2} }
\| \theta \|_{ C^{0} } 
\Big)^{1-t}
\Big(
( 2^{j} )^{ \frac{n}{2} } \,
\min \big\{ 2^{j} r, ( 2^{j} r )^{ -\frac{n}{2} } \big\} \,
\| \theta \|_{ C^{0} } 
\Big)^{t}
\\
&
=
( 2^{j} )^{ \frac{n}{2} } \,
\min \big\{ (2^{j} r)^{t} , ( 2^{j} r )^{ -\frac{nt}{2} } \big\}
\| \theta \|_{ C^{0} } ,
\end{align*}
which completes the proof of
the assertion \eqref{Sj_L2_Rn}.

\eqref{Sj_L2_torus}
We observe that,
if $A \ge 2r$, 
$A \le |x| \le 2 A$,
and $|y| \le r \le \frac{A}{2}$,
then
$|x-y| \approx |x| \approx A$.
Since $f$ is an $h^1$-atom 
supported on a ball centered at the origin,
this observation yields that
\begin{align*}
\big\| S_{j} f(x) &
\big\|_{ L^2 (A \le |x| \le 2A) }
=
\Big\|
\big( 2^{j(1-s)} |x| \big)^{-N}
\big( 2^{j(1-s)} |x| \big)^{N} \,
K_{j} \ast f (x)
\Big\|_{ L^2 (A \le |x| \le 2A) }
\\
&
\lesssim
\big( 2^{j(1-s)}  {A} \big)^{-N}
\Big\|
\int_{|y| \le r} 
\big( 2^{j(1-s)} |x-y| \big)^{N}
\big| K_{j}(x-y)\big| \, 
|f(y)| 
\,dy
\Big\|_{ L^2_x (\R^n) } 
\\
&
\lesssim
\big( 2^{j(1-s)}  {A} \big)^{-N}
\Big\|
\big( 2^{j(1-s)} |x| \big)^{N}
 K_{j}(x)
 \Big\|_{L^2}.
\end{align*}
Here,
we recall \eqref{K_j_no_weighedL^2}.
Since the assumption in this assertion
is identical with \eqref{N_no_condition},
by utilizing \eqref{K_j_no_weighedL^2},
we obtain
\begin{align*}
\big\| S_{j} f(x) \big\|_{ L^2 (A \le |x| \le 2A) }
\lesssim
\| \theta \|_{ C^{M} } \,
( 2^{j} )^{ \frac{n}{2} } \,
\big( 2^{j(1-s)} A \big)^{ -N }
\end{align*}
for some constant $M \in \N$.
Also,
the inequality \eqref{L2_Rn_t=1}
obviously holds
if $L^2 (\R^n)$
is replaced by
$L^2 (A \le |x| \le 2 A) $.
Therefore,
interpolating these two inequalities,
we have for $0 \le t \le 1$
\begin{align*}
&
\big\| S_{j} f(x) \big\|_{  {L^2 (A \le |x| \le 2 A)} }
=
\Big(
\| S_{j} f(x) \|_{  {L^2 (A \le |x| \le 2 A)} }
\Big)^{1-t}
\Big(
\| S_{j} f(x) \|_{  {L^2 (A \le |x| \le 2 A)} }
\Big)^{t}
\\
&
\lesssim
\Big(
\| \theta \|_{ C^{M} } \,
( 2^{j} )^{ \frac{n}{2} } \,
\big( 2^{j(1-s)}  {A} \big)^{ -N }
\Big)^{1-t}
\Big(
\| \theta \|_{ C^{0} } \,
( 2^{j} )^{ \frac{n}{2} } \,
\min \big\{ 2^{j} r, ( 2^{j} r )^{ -\frac{n}{2} } \big\}
\Big)^{t}
\\
&
\le
\| \theta \|_{ C^{M} } \,
( 2^{j} )^{ \frac{n}{2} } \,
\big( 2^{j(1-s)}  {A} \big)^{ -N(1-t) }
\min \big\{ (2^{j} r)^{t} , ( 2^{j} r )^{ -\frac{nt}{2} } \big\},
\end{align*}
which completes the proof of 
the assertion \eqref{Sj_L2_torus}.
\end{proof}

\begin{lem} \label{T}
Let $0< s <\infty$.  
Then 
there exist $c>0$ and $M \in \N$ depending only on 
$n$ and $s$ 
such that 
the following hold.
\begin{enumerate}
\item \label{T_Lq}
If $1 \le p \le q \le \infty$, then 
\[
\| T f \|_{ L^{q} ( \R^{n} ) }
\le c\,
\|\theta\|_{ C^{M} }
\|f\|_{ L^{p} (\R^{n}) } .
\]

\item \label{T_Linfty}
If $f$ is 
an $h^1$-atom
supported on 
a ball of radius $r$ 
centered at the origin in $\R^n$ 
and if $A \ge 2r$,
then 
\[
\| Tf \|_{  L^\infty (A \le |x| \le 2 A) }
\le c\, 
A^{ -(n+s) }
\|\theta\|_{ C^{M} }. 
\]
\end{enumerate}
\end{lem}

\begin{proof}
Before beginning with proofs of the assertions,
we show that the kernel $L$ defined in \eqref{kernels}
satisfies the following inequality:
there exists $M \in \N$
such that
\begin{align} \label{Linq}
\big| L (x) \big|
\lesssim 
\| \theta \|_{ C^{M} }
( 1+ |x| )^{-(n+s)} .
\end{align}
Although the inequality \eqref{Linq}
is a well-known fact,
for the sake of a self-contained proof,
we revisit a proof.

The case for $|x| \le 1$ 
is simple, and so we will consider 
the case $|x| \ge 1$.
Since 
$e^{i|\xi|^s} -1 = i |\xi|^{s} \int_{0}^{1} e^{ it |\xi|^s } \,dt$,
the kernel $L$ can be expressed by
\begin{align*}&
L (x)
=
\int_{ \substack{ |\xi| \le 2 \\ 0 < t < 1 } }
e^{i x \cdot \xi}
e^{i t |\xi|^s} \,
\big( i |\xi|^{s} \big)
\theta (\xi)
\, d\xi dt
+
\int_{ |\xi| \le 2 }
e^{i x \cdot \xi}
\theta (\xi) 
\, d\xi .
\end{align*}
Integration by parts yields that 
the absolute value of the second integral 
is bounded by 
$\| \theta \|_{ C^{M} } (1+|x|)^{-M}$
for any $M \in \N_0$,
and thus,
in the following, 
we shall consider
the first integral.
Using a Littlewood--Paley partition of unity on $\R^n$,
$\{ \psi (2^{-k} \cdot) \}_{ k \in \Z }$,
since 
$\supp \theta \subset
\{ |\xi| \le 2 \}$,
we can decompose the first integral into
\begin{align*}
\sum_{ k \le 1 }
I_{k} (x)
\quad\textrm{with}\quad
I_{k} (x)
:=
\int_{ \substack{ 2^{k-1} \le |\xi| \le 2^{k+1} \\ 0 < t < 1 } }
e^{i x \cdot \xi}
e^{i t |\xi|^s} \,
\big( i |\xi|^{s} \big)
\psi (2^{-k} \xi )
\theta (\xi)
\, d\xi dt .
\end{align*}
Here,
we have
for $\alpha \in (\N_{0})^{n}$ and $k \le 1$
\[
\Big|
\partial_{\xi}^{\alpha}
\Big(
e^{i t |\xi|^s} \,
\big( i |\xi|^{s} \big)
\psi (2^{-k} \xi )
\theta (\xi) 
\Big)
\Big|
\lesssim
\| \theta \|_{ C^{|\alpha|} }
(2^{k})^{s-|\alpha|}
\ichi_{ \{ 2^{k-1} \le |\xi| \le 2^{k+1} \} } ,
\quad
0 \le t \le 1,
\]
which gives that for $M \in \N_{0}$
\begin{align*}
\big| I_{k} (x) \big| \lesssim
\| \theta \|_{ C^{M} }
\times
\begin{cases}
( 2^{k} )^{n+s} = |x|^{-(n+s)} \, ( 2^{k} |x| )^{n+s} , 
\\
|x|^{-M} ( 2^{k} )^{-M+n+s} = |x|^{-(n+s)} \, ( 2^{k} |x| )^{-M+n+s} .
\end{cases}
\end{align*}
Therefore,
by choosing $M > n+s$,
\begin{align*}
\sum_{k \le 1}
| I_{k} (x) |
&
\lesssim
\| \theta \|_{ C^{M} }
|x|^{-(n+s)}
\sum_{k \le 1}
\min \{
( 2^{k} |x| )^{n+s} ,
( 2^{k} |x| )^{-M+n+s}
\}
\approx 
\| \theta \|_{ C^{M} }
|x|^{-(n+s)} ,
\end{align*}
which completes
the proof of \eqref{Linq}.
Now, we actually prove 
the assertions \eqref{T_Lq} and \eqref{T_Linfty}.

\eqref{T_Lq}
Take a function $\widetilde{\theta} \in \calS(\R^n)$
satisfying that
$\widetilde{\theta} = 1$ on 
$\{ |\xi| \le 2 \}$
and 
$\supp \widetilde{\theta} \subset 
\{ |\xi| \le 3\}$.
Then, 
we observe that
\[
T f (x)
=
T (\widetilde{\theta}(D) f) (x)
\]
and also
from \eqref{Linq} that $L$,  
the kernel of $T$,
is in $L^1 (\R^n)$
and $\|L\|_{L^1} \lesssim \|\theta\|_{C^M}$.
Therefore, we see that
\begin{align*}
\| Tf \|_{L^q}
&
\lesssim
\| \theta \|_{ C^{M} }
\| \widetilde{\theta}(D) f \|_{L^q}
\\
&
\lesssim
\| \theta \|_{ C^{M} }
\| \widetilde{\theta}(D) f \|_{L^p}
\lesssim
\| \theta \|_{ C^{M} }
\| f \|_{L^p },
\end{align*}
where,
in the second inequality,
we used Nikol'skij's inequality (see, {\it e.g.}, 
\cite[Section 1.3.2, Remark 1]{triebel 1983}).
This completes the proof of 
the assertion \eqref{T_Lq}.

\eqref{T_Linfty}
We first observe that,
if $A \ge 2r$,  
$A \le |x| \le 2A$,
and $|y| \le r \le \frac{A}{2}$,
then
$|x-y| \approx |x| \approx A$.
By \eqref{Linq},
we have for 
$A \le |x| \le 2 A$
\begin{align*}
\big| T f(x) \big|
&
\le
\int_{|y| \le r} 
\big| L(x-y)\big| \, 
|f(y)| 
\,dy 
\lesssim
\|\theta\|_{ C^{M} }
\int_{|y| \le r} 
|x-y|^{-(n+s)}
|f(y)| 
\,dy 
\\
&
\lesssim
A^{- (n+s) }
\|\theta\|_{ C^{M} } ,
\end{align*}
which implies 
the assertion \eqref{T_Linfty}.
\end{proof}

\subsection{Inequalities for $s < 1$}

In this subsection, 
we show some inequalities 
which will be used
for proving the boundedness
for $s<1$. 

\begin{lem} \label{Sj_s<1}
Let $0< s < 1$.
Then 
there exist $c>0$ and $M \in \N$ depending only on 
$n$ and $s$ 
such that 
the following hold. 
\begin{enumerate}
\item 
\label{Sj_s<1_outside}
If $f$ is an $h^1$-atom 
supported on 
a ball 
centered at the origin in $\R^n$,  
then
\[
\big\| S_{j} f(x) \big\|_{ L^1(|x| \ge 2)}
\le c\,
\| \theta \|_{C^M} 
\]
for all $j \in \N_{0}$. 

\item
\label{Sj_s<1_inside}
If $0 < A \le 10$, then
\[
\big\| S_{j} f (x)
\big\|_{ L^{2} (|x| \le A) }
\le
c
A^{ \frac{n(1-s)}{2} } \,
\|\theta\|_{ C^{M} } \,
\|f\|_{ L^\infty (\R^n) } 
\]
for all $j \in \N_{0}$. 
\end{enumerate}

\end{lem}

\begin{proof}
\eqref{Sj_s<1_outside}
If $|x| \ge 2$ and $|y| \le r \le 1$, then $|x-y| \ge \frac{|x|}{2} \ge 1$.
Hence, 
by Corollary \ref{Kj_phi} \eqref{Kj_phi_s<1}, 
we obtain for $N>n$
\begin{align*}
\|S_jf(x)\|_{L^1(|x| \ge 2)}
&
\lesssim
\|\theta\|_{C^M}
\int_{|x| \ge 2}
\int_{|y| \le r}
|x-y|^{-N}
|f(y)|
\,
dy
dx
\\
&
\lesssim
\|\theta\|_{C^M}
\Big(
\int_{|x| \ge 2}
|x|^{-N}
\,
dx
\Big)
\int_{|y| \le r}
|f(y)|
dy
\lesssim
\|\theta\|_{C^M},
\end{align*}
which completes the proof of 
the assertion \eqref{Sj_s<1_outside}.
(We don't need the moment condition $\int f = 0$ here.)

\eqref{Sj_s<1_inside}
We decompose $f$ by
\begin{align*}
f
=
f \ichi_{ \{ |y| \le C A^{1-s} \} }
+
f \ichi_{ \{ |y| > C A^{1-s} \} }
=:
f_{A}^{1} + f_{A}^{2},
\end{align*}
where $C = 2 \cdot 10^s$.
For the estimate with respect to $f_{A}^{1}$,
we see from Plancherel's theorem that
\begin{align*}
\| S_{j} f_{A}^{1} \|_{ L^2 (|x|\le A)}
\le 
\| S_{j} f_{A}^{1} \|_{ L^2 (\R^n)}
\le 
\| \theta \|_{ C^{0} }
\| f_{A}^{1} \|_{ L^2 (\R^n)}
\lesssim
\| \theta \|_{ C^{0} } \,
A^{ \frac{n(1-s)}{2} } \,
\| f \|_{ L^{\infty} (\R^n)} .
\end{align*}
We next consider the estimate with respect to $f_{A}^{2}$.
In the situation here,
since $A \le 10^{s} A^{1-s}$
for $0< A \le 10$ and $0< s <1$,
we realize that,
if $|x| \le A$ and $|y| \ge C A^{1-s}$,
then
\[
|x-y| 
\ge 
\Big(1- \frac{10^s}{C} \Big)|y| 
=
\frac{|y|}{2}.
\]
Hence,
by Corollary \ref{Kj_phi} \eqref{Kj_phi_s<1},
\begin{align*}
&
\| S_{j} f_{A}^{2} \|_{ L^2 (|x|\le A)}
\lesssim
A^{ \frac{n}{2} }
\Big\|
\int_{|y| > C A^{1-s}} 
\big| K_{j}(x-y) \big| \, 
|f(y)|
\,dy 
\Big\|_{ L^{\infty}_{x} (|x|\le A)}
\\
&
\lesssim
A^{ \frac{n}{2} } \,
\|\theta\|_{ C^{M} }
\|f\|_{ L^{\infty} (\R^n)}
\int_{|y| > C A^{1-s}} 
|y|^{ -\frac{n}{2} - \frac{n}{2(1-s)} } \, 
\,dy 
\approx
A^{ \frac{n(1-s)}{2} } \,
\|\theta\|_{ C^{M} }
\| f \|_{ L^{\infty} (\R^n)}
,
\end{align*}
where,
in the last inequality,
we used $ -\frac{n}{2} - \frac{n}{2(1-s)} < -n$.
This completes the proof.
\end{proof}

\subsection{Inequalities for $s > 1$}

In this subsection, 
we show some inequalities 
which will be used
for proving the boundedness
in the case $s>1$. 

\begin{lem} \label{Sj_s>1}
Let $1 < s < \infty$. 
If $j \in \N_{0}$ and $A \ge 2^{j(s-1)}$, then 
\[
\big\| S_{j} f (x)
\big\|_{ L^{2} (|x| \le A) }
\le
c
A^{ \frac{n}{2} } \,
\|\theta\|_{ C^{M} } \,
\|f\|_{ L^\infty (\R^n) }, 
\]
where the constants 
$c>0$ and $M\in \N$ depend only on 
$n$ and $s$.
\end{lem}

\begin{proof}
We decompose $f$ as follows:
\begin{align*}
f
=
f \ichi_{ \{ |y| \le C A \} }
+
f \ichi_{ \{ |y| > C A \} }
=:
f_{A}^{1} + f_{A}^{2} ,
\end{align*}
where $C = 2s 8^{s-1}$.
For the estimate involved in $f_{A}^{1}$,
we have by Plancherel's theorem
\begin{align*}
\| S_{j} f_{A}^{1} \|_{ L^2 (|x|\le A)}
\le 
\| S_{j} f_{A}^{1} \|_{ L^2 (\R^n)}
\le 
\| \theta \|_{ C^{0} }
\| f_{A}^{1} \|_{ L^2 (\R^n)}
\lesssim
\| \theta \|_{ C^{0} } \,
A^{ \frac{n}{2} }\,
\| f \|_{ L^{\infty} (\R^n)} ,
\end{align*}
Next, we consider
the estimate involved in $f_{A}^{2}$.
Observe that,
for $|x| \le A$ and $|y| \ge C A$,
\[
|x-y| 
\ge 
\Big(1- \frac{1}{C} \Big)|y| 
\ge
\frac{|y|}{2}
\ge
\frac{CA}{2}
\ge 
s 8^{s-1} 2^{j(s-1)} 
\]
holds,
which is possible from the choice 
of the constant $C \ge 2$.
Hence, by utilizing
Corollary \ref{Kj_phi} \eqref{Kj_phi_s>1},
we have for large $N>n$
\begin{align*}
&
\| S_{j} f_{A}^{2} \|_{ L^2 (|x|\le A)}
\lesssim
A^{ \frac{n}{2} }
\Big\|
\int_{|y| > C A} 
\big| K_{j}(x-y) \big| \, 
|f(y)|
\,dy 
\Big\|_{ L^{\infty}_{x} (|x|\le A)}
\\
&
\lesssim
A^{ \frac{n}{2} } \,
\|\theta\|_{ C^{M} }
\| f \|_{ L^{\infty} (\R^n)}
\int_{|y| > C A } 
|y|^{ -N } \, 
\,dy 
\lesssim
A^{ \frac{n}{2} } \,
\|\theta\|_{ C^{M} }
\| f \|_{ L^{\infty} (\R^n)} .
\end{align*}
Combining the above estimates,
we complete the proof of this lemma.
\end{proof}

\section{Boundedness in $H^1 \times L^\infty \to L^1$}
\label{Bdd_H1LinftyL1}

In this section, we shall give a proof of Theorem \ref{thm_main}. 
To this end, we will prove the following theorem.

\begin{thm} \label{thm_h1LinftyL1}
Let $0 < s < 1$ or $1 < s < \infty$.
Suppose that
$\sigma \in S_{1,0}^{m} (\R^{2n})$
with
\[
m = 
\begin{cases}
 -\frac{sn}{2} - \frac{s(1-s)n}{2} ,
	& \textrm{if} \;\; 0 < s < 1,
\\
-\frac{sn}{2} ,
	& \textrm{if} \;\; 1 < s < \infty.
\end{cases}
\]
Then
$T_{\sigma}^{s}$ is bounded from $h^{1} \times L^{\infty}$ to $L^{1}$. 
\end{thm}

We notice that 
Theorem \ref{thm_main} can be derived from 
Theorems \ref{thm_BRRS} and \ref{thm_h1LinftyL1}
by virtue of complex interpolation.
Thus, it suffices to show Theorem \ref{thm_h1LinftyL1}.

Now, we begin with the proof of Theorem \ref{thm_h1LinftyL1}.
We decompose the multiplier $\sigma$ following 
the idea of Coifman-Meyer \cite{CM-Ast, CM-AIF}.
We write
\begin{align*} &
\sigma (\xi, \eta) 
=
\sum_{j=0}^{\infty} 
\sum_{k=0}^{\infty} 
\sigma (\xi, \eta) \psi_j (\xi) \psi_k (\eta)
\\&=
\sigma (\xi, \eta) \varphi (\xi) \varphi (\eta)
+
\sum_{j=1}^{\infty} 
\sum_{k=0}^{j} 
\sigma (\xi, \eta) \psi_j (\xi) \psi_k (\eta)
+
\sum_{k=1}^{\infty} 
\sum_{j=0}^{k-1} 
\sigma (\xi, \eta) \psi_j (\xi) \psi_k (\eta)
\\&=
\sigma (\xi, \eta) \varphi (\xi) \varphi (\eta)
+
\sum_{j \in \N} 
\sigma (\xi, \eta) \psi_{j} (\xi) \varphi_{j} (\eta)
+
\sum_{k \in \N} 
\sigma (\xi, \eta) \varphi_{k-1} (\xi) \psi_k (\eta)
\\&=
\sigma_{0} (\xi, \eta)
+
\sigma_{\RomI} (\xi, \eta)
+
\sigma_{\II} (\xi, \eta) .
\end{align*}

We first consider the multiplier $\sigma_{\RomI}$. 
Taking functions 
$\widetilde{\psi}, \widetilde{\varphi} \in C_{0}^{\infty}(\R^n)$ 
such that 
\[\begin{array}{ll}
\widetilde{\psi} = 1 
\;\;\text{on}\;\; \{ 2^{-1}\le |\xi| \le 2 \}, \quad
&
\supp \widetilde{\psi} 
\subset \{3^{-1}\le |\xi| \le 3\}, 
\vspace{3pt}\\
\widetilde{\varphi} = 1 
\;\;\text{on}\;\; \{ |\xi| \le 2 \}, \quad
&
\supp \widetilde{\varphi} 
\subset \{|\xi| \le 3\} ,
\end{array}\]
we can write $\sigma_{\RomI}$ as
\[
\sigma_{\RomI}(\xi, \eta)
=
\sum_{j \in \N }
\sigma (\xi, \eta) 
\widetilde{\psi}(2^{-j}\xi) 
\widetilde{\varphi}(2^{-j}\eta) \,
\psi_{j} (\xi) 
\varphi_{j} (\eta),
\]
since 
$\widetilde{\psi}( 2^{-j} \xi ) 
\widetilde{\varphi}( 2^{-j} \eta)$
equals 1
on the support of
$\psi_{j} (\xi) \varphi_{j} (\eta)$.
Since $\sigma \in S_{1,0}^{m}(\R^{2n})$ and
\[
\supp \sigma (2^{j} \xi, 2^{j}\eta)\widetilde{\psi}(\xi) 
\widetilde{\varphi}(\eta)
\subset 
\{3^{-1}\le |\xi|\le 3\}\times 
\{|\eta|\le 3\},
\]
the following estimate holds:
\begin{equation*}
\left|
\partial^{\alpha}_{\xi} 
\partial^{\beta}_{\eta} 
\big(
\sigma (2^{j} \xi, 2^{j}\eta)\widetilde{\psi}(\xi) 
\widetilde{\varphi}(\eta)
\big) 
\right| 
\le 
C_{\alpha, \beta}\,  2^{jm}
\end{equation*}
with $C_{\alpha, \beta}$ independent of $j\in \N_0$. 
Hence, by the Fourier series expansion
with respect to the variables $\xi$ and $\eta$,
we can write 
\[
\sigma (2^{j} \xi, 2^{j}\eta)
\widetilde{\psi}(\xi) 
\widetilde{\varphi}(\eta)
=
\sum_{a, b \in \Z^n} 
c_{\RomI, j}^{(a, b)}
e^{i a \cdot \xi} e^{i b \cdot \eta}, 
\quad 
|\xi|<\pi, \;\; |\eta|<\pi,  
\]
with the coefficient satisfying that
for $L>0$
\begin{equation} \label{cjab-decay}
\big| c_{\RomI, j}^{(a, b)} \big| 
\lesssim 2^{jm} 
(1+|a|)^{-L} (1+|b|)^{-L}.
\end{equation}
Changing variables 
$\xi \to 2^{-j}\xi$ and $\eta \to 2^{-j}\eta$
and 
multiplying $\psi_{j} (\xi) \varphi_{j} (\eta)$, 
we obtain 
\[
\sigma (\xi, \eta)
\psi_{j} (\xi) 
\varphi_{j} (\eta)
=
\sum_{a, b \in \Z^n} 
c_{\RomI, j}^{(a, b)} \,
e^{i a \cdot 2^{-j} \xi} 
e^{i b \cdot 2^{-j}\eta}
\psi_{j} (\xi)
\varphi_{j} (\eta).  
\]
Hence, by the definitions of
$\psi_{j}$ and $\varphi_{j}$
in Notation \ref{notation}, 
the multiplier $\sigma_{\RomI}$ is written as
\begin{align*}
\sigma_{\RomI} (\xi, \eta)
&=
\sum_{a, b \in \Z^n}
\sum_{j \in \N }
c_{\RomI, j}^{(a, b)} \,
e^{i a \cdot 2^{-j} \xi} e^{i b \cdot 2^{-j}\eta}
\psi (2^{-j} \xi) \,
\varphi (2^{-j} \eta) 
\\
&
=
\sum_{a, b \in \Z^n}
\sum_{j \in \N }
c_{\RomI, j}^{(a, b)} \,
\psi^{(a)} (2^{-j} \xi) \,
\varphi^{(b)} (2^{-j} \eta) ,
\end{align*}
where we wrote as
\begin{equation*}
\psi^{(\nu)} (\xi)
= 
e^{i \nu \cdot \xi} 
\psi (\xi), 
\quad  
\varphi^{(\nu)} (\eta) 
= 
e^{i \nu \cdot \eta} 
\varphi (\eta),
\quad
\nu \in \Z^n .
\end{equation*}

By similar arguments, 
the multipliers
$\sigma_{0}$ and
$\sigma_{\II}$
can be written as
\begin{align*}
\sigma_{0} (\xi, \eta)
&
=
\sum_{a, b \in \Z^n}
c_{0}^{(a, b)} \,
\varphi^{(a)} (\xi) \,
\varphi^{(b)} (\eta) ,
\\
\sigma_{\II} (\xi, \eta)
&
=
\sum_{a, b \in \Z^n}
\sum_{j \in \N }
c_{\II, j}^{(a, b)} \,
\varphi^{(a)} (2^{-(j-1)} \xi) \,
\psi^{(b)} (2^{-j} \eta) ,
\end{align*}
where 
the coefficient
$c_{0}^{(a, b)}$ 
satisfies the same condition as in \eqref{cjab-decay} with $j=0$,
and 
the coefficient 
$c_{\II, j}^{(a, b)}$ 
satisfies the same condition as in \eqref{cjab-decay}.

Hereafter we shall consider 
slightly general multipliers 
$\widetilde{\sigma}_0$ and $\widetilde{\sigma}$
defined by 
\begin{align}&\label{assumption-sigma_0}
\widetilde{\sigma}_0 (\xi, \eta) 
=
c_{0} \,
\theta_{1} (\xi) \,
\theta_{2} (\eta) ,
\\
&\label{assumption-sigma}
\widetilde{\sigma} (\xi, \eta) 
=
\sum_{j \in \N}
c_{j} \,
\theta_{1} (2^{-j} \xi) \,
\theta_{2} (2^{-j} \eta), 
\end{align}
where 
$( c_j )_{j\in \N_0}$ is 
a sequence of complex numbers satisfying 
\begin{equation}\label{assumption-cj}
\left| c_{j} \right| 
\le 2^{jm} A, 
\quad 
j\in \N_0, 
\end{equation}
with some $A\in (0, \infty)$,  
and 
$\theta_1, \theta_2 \in \calS(\R^n)$ 
satisfy that 
\begin{equation}\label{assumption-theta}
\supp \theta_1 , \, \supp \theta_2 
\subset 
\{|\xi|\le 2 \} .
\end{equation}
For such $\widetilde{\sigma}_0$ and $\widetilde{\sigma}$, 
we shall prove that
there exist 
$c > 0$ and $M \in \N$
such that
\begin{equation}\label{Goal-estimate}
\| T^{s}_{ \widetilde{\sigma}_0 }
\|_{h^1 \times L^{\infty}\to L^1}, 
\;
\| T^{s}_{ \widetilde{\sigma} }
\|_{h^1 \times L^{\infty}\to L^1}
\le 
c A \|\theta_1\|_{ C^{M} } 
\|\theta_2\|_{ C^{M} }  .
\end{equation}

If this is proved, by applying
\eqref{Goal-estimate} for $\widetilde{\sigma}$ to 
$c_j = c_{\RomI, j}^{(a, b)}$,  
$\theta_{1} = \psi^{(a)}$, and 
$\theta_{2} = \varphi^{(b)}$,
we have 
\begin{align*}
\| T^{s}_{ {\sigma}_{\RomI}}\|_{ h^1 \times L^{\infty}\to L^1 }
&
\lesssim 
\sum_{a, b \in \Z^n}
(1+|a|)^{-L} (1+|b|)^{-L} 
\|e^{i a \cdot \xi} \psi (\xi)\|_{C^{M}_{\xi}} 
\|e^{i b \cdot \eta} \varphi (\eta)\|_{C^{M}_{\eta}} 
\\
&
\lesssim 
\sum_{a, b \in \Z^n}
(1+|a|)^{-L+M}
(1+|b|)^{-L+M}, 
\end{align*}
and, thus, taking $L$ sufficiently large,
we see that 
$T^{s}_{ {\sigma}_{\RomI} }$ is 
bounded from $h^1 \times L^{\infty}$ to $L^1$.
In the same way, 
we see that 
$T^{s}_{ {\sigma}_{0} }$ and $T^{s}_{ {\sigma}_{\II} }$
are bounded from $h^1 \times L^{\infty}$ to $L^1$.
The above three estimates complete 
the proof of the $h^1 \times L^{\infty} \to L^1$ boundedness 
of $T_{\sigma}^{s}$.

Thus the proof is reduced to showing \eqref{Goal-estimate} 
for $\widetilde{\sigma}_0$ and $\widetilde{\sigma}$ 
given by \eqref{assumption-sigma_0}--\eqref{assumption-theta}.
However, the estimate for $\widetilde{\sigma}_0$ is simple.
In fact, from 
Lemma \ref{T} \eqref{T_Lq},
\begin{align*}&
\| T^{s}_{ \widetilde{\sigma}_0 } (f,g) \|_{ L^1}
= 
\Big\| c_{0} \,
e^{i|D|^{s}} \theta_{1}(D) f (x) \,
e^{i|D|^{s}} \theta_{2}(D) g (x) 
\Big\|_{ L^1 }
\\
&
\le
A
\big\| 
e^{i|D|^{s}} \theta_{1}(D) f \big\|_{ L^1}
\big\|
e^{i|D|^{s}} \theta_{2}(D) g
\big\|_{ L^{\infty} }
\lesssim
A \|\theta_{1}\|_{ C^{M} }
\|\theta_{2}\|_{ C^{M} }
\|f\|_{ L^1 } \|g\|_{ L^{\infty} } .
\end{align*}
Hence, 
in what follows,
we concentrate on proving
\eqref{Goal-estimate} 
for $\widetilde{\sigma}$ 
given by \eqref{assumption-sigma}--\eqref{assumption-theta}.

We shall make further reductions. 
Using $\varphi$ and $\zeta$ 
defined in Notation \ref{notation}, 
we decompose the function $1$ on $\R^n \times \R^n$
into
\begin{align*}
&
1
=
\varphi (\xi) 
\varphi (\eta)  
+ 
\varphi (\xi)
\zeta (\eta) 
+ 
\zeta (\xi)
\varphi (\eta)
+
\zeta (\xi) 
\zeta (\eta) .  
\end{align*}
Then,
$T^{s}_{ \widetilde{\sigma} }$
can be expressed by the following four parts:
\begin{align} \label{redTsigma}
\begin{split}
T^{s}_{ \widetilde{\sigma} }(f,g) (x)
=
\sum_{ j \in \N }
c_{j} \,
\Big\{
&
T_{j}^{1} f (x) \,
T_{j}^{2} g (x)
+
T_{j}^{1} f (x) \,
S_{j}^{2} g (x)
\\
&
+
S_{j}^{1} f (x) \,
T_{j}^{2} g (x)
+
S_{j}^{1} f (x) \,
S_{j}^{2} g (x)
\Big\}
,
\end{split}
\end{align}
where,
for $\ell = 1,2$,
we wrote 
\begin{align*}
S_{j}^{\ell} f (x) 
&
=
\Big( 
e^{ i |\xi|^{s} } \, 
\zeta(\xi) \theta_{\ell} (2^{-j} \xi ) \,
\widehat{f} (\xi) 
\Big)^{\vee} (x)
,
\\
T_{j}^{\ell} f (x)
&
=
\Big( 
e^{ i |\xi|^{s} } \, 
\varphi(\xi) \theta_{\ell} (2^{-j} \xi ) \,
\widehat{f} (\xi) 
\Big)^{\vee} (x) .
\end{align*}
Considering
the $L^1$-norm of \eqref{redTsigma}
and using the assumption \eqref{assumption-cj},
we see that
\begin{align*} 
\big\|
T^{s}_{ \widetilde{\sigma} }(f,g)
\big\|_{ L^1 }
\le
A
\sum_{U,V \in \{S,T\}}
\sum_{ j \in \N }
2^{jm}
&
\big\| U_{j}^{1} f \,
V_{j}^{2} g \big\|_{L^1} ,
\end{align*}
and thus,
in the following argument,
it is sufficient to prove that
\[
\sum_{ j \in \N }
2^{jm} 
\big\| U_{j}^{1} f \,
V_{j}^{2} g \big\|_{L^1}
\lesssim
\|\theta_{1}\|_{ C^{M} }
\|\theta_{2}\|_{ C^{M} }
\|f\|_{ h^1 } \|g\|_{ L^{\infty} },
\quad
U,V \in \{S, T\} .
\]
To prove this, 
by virtue of the atomic decomposition of $h^1$,
stated in Notation \ref{notation},
and by translation invariance, 
it suffices to obtain the uniform estimates 
for $h^1$-atoms $f$
supported on balls
centered at the origin.
Furthermore, 
we may assume that $\|g\|_{ L^\infty } = 1$.
Therefore,
in order to obtain the desired boundedness result 
in Theorem \ref{thm_h1LinftyL1},
we shall prove that
\begin{equation} \label{goal}
\sum_{ j \in \N }
2^{jm} 
\big\| U_{j}^{1} f \,
V_{j}^{2} g \big\|_{L^1}
\lesssim 
\|\theta_1\|_{C^{M}}
\|\theta_2\|_{C^{M}} ,
\quad
U,V \in \{S, T\} ,
\end{equation}
holds for such
$f$ and $g$
and for some $M \in \N$.
In the rest of this section, the letter $r$ always denotes
the radius of a ball including the support of $f$;
$f$ is assumed to be an $h^1$-atom supported 
on a ball in $\R^n$ of radius $r$ centered at the origin. 

\subsection{Case $0<s<1$}
\label{0<s<1}

We recall that the critical order for $0< s < 1$ is 
\begin{align*}
m= -\frac{sn}{2} - \frac{s(1-s)n}{2},
\end{align*}
and notice that $m < - \frac{sn}{2}$.
We also remark that this $m$ can be written as 
\begin{equation*}
m= - \frac{n}{2} + \frac{(1-s)^2n}{2}.
\end{equation*}

\subsubsection{Estimate for $T_j^1f\, T_j^2g$}

By 
Lemma \ref{T} \eqref{T_Lq},
it holds that
\begin{align*}
\big\|
T_j^1f\, T_j^2g
\big\|_{L^1}
&
\le
\big\|
T_j^1f
\big\|_{L^1}
\big\|
T_j^2g
\big\|_{L^\infty}
\\
&
\lesssim
\|\varphi(\cdot) \, \theta_{1} (2^{-j} \cdot)\|_{ C^{M} }
\|\varphi(\cdot) \, \theta_{2} (2^{-j} \cdot)\|_{ C^{M} }
\|f\|_{L^1} 
\|g\|_{L^\infty}
\\
&
\lesssim
\|\theta_1\|_{C^M}
\|\theta_2\|_{C^M}.
\end{align*}
Since $m < 0$, we obtain \eqref{goal} with $(U, V) = (T, T)$ from this estimate.

\subsubsection{Estimate for $T_j^1f\, S_j^2g$}

It follows from 
Lemmas \ref{T} \eqref{T_Lq}
and \ref{Sj_Lp}
that 
\begin{align*}
\big\|
T_j^1f\, S_j^2g
\big\|_{L^1}
&
\le
\big\|
T_j^1f
\big\|_{L^1}
\big\|
S_j^2g
\big\|_{L^\infty}
\\
&
\lesssim
(2^{j})^{\frac{sn}{2}}
\|\varphi(\cdot) \, \theta_{1} (2^{-j} \cdot)\|_{ C^{M} }
\|\theta_2\|_{C^M}
\|f\|_{L^1} 
\|g\|_{L^\infty}
\\
&
\lesssim
(2^{j})^{\frac{sn}{2}}
\|\theta_1\|_{C^M}
\|\theta_2\|_{C^M},
\end{align*}
which gives \eqref{goal} with $(U, V) = (T, S)$ because $m < -\frac{sn}{2}$.

\subsubsection{Estimate for $S_j^1f\, T_j^2g$}
We use 
Lemmas \ref{Sj_Lp}
and \ref{T} \eqref{T_Lq}
to obtain
\begin{align*}
\big\|
S_j^1f\, T_j^2g
\big\|_{L^1}
&
\le
\big\|
S_j^1f
\big\|_{L^1}
\big\|
T_j^2g
\big\|_{L^\infty}
\\
&
\lesssim
(2^{j})^{\frac{sn}{2}}
\|\theta_1\|_{C^M}
\|\varphi(\cdot) \, \theta_{2} (2^{-j} \cdot)\|_{ C^{M} }
\|f\|_{L^1} 
\|g\|_{L^\infty}
\\
&
\lesssim
(2^{j})^{\frac{sn}{2}}
\|\theta_1\|_{C^M}
\|\theta_2\|_{C^M}.
\end{align*}
Since $m < -\frac{sn}{2}$, this yields that \eqref{goal} holds with $(U, V) = (S, T)$.

\subsubsection{Estimate for $S_j^1f\, S_j^2g$}

We divide the $L^1$ norm in \eqref{goal} into the following three parts;
\begin{align} \label{decomp_first_kind}
\big\|
S_j^1f\, S_j^2g
\big\|_{L^1(\R^n)}
=
\big\|
S_j^1f\, S_j^2g
\big\|_{L^1(|x| \le 2r)}
+
\big\|
S_j^1f\, S_j^2g
\big\|_{L^1(2r < |x| \le 4)}
+
\big\|
S_j^1f\, S_j^2g
\big\|_{L^1(|x| > 4)}.
\end{align}
We first consider the norm ${L^1(|x| \le 2r)}$.
By the Cauchy-Schwarz inequality, we obtain
\begin{align*}
\big\|
S_{j}^{1} f
\big\|_{L^1(|x| \le 2r)}
\lesssim
r^{n/2}
\big\|
S_{j}^{1} f
\big\|_{L^2(\R^n)}
\lesssim
r^{n/2}
\|\theta_1\|_{C^{ 0 }}
\|f\|_{L^2}
\lesssim
\|\theta_1\|_{C^{ 0 }}.
\end{align*}
Hence,
by this inequality and Lemma \ref{Sj_Lp}, 
we have
\begin{align} \label{L^1_inside}
&
\big\|
S_j^1f\, S_j^2g
\big\|_{L^1(|x| \le 2r)}
\le
\big\|
S_j^1f
\big\|_{L^1(|x| \le 2r)}
\big\|
S_j^2g
\big\|_{L^{\infty}(\R^n)}
\lesssim
(2^j)^{\frac{sn}{2}}
\|\theta_1\|_{C^0}
\|\theta_2\|_{C^M}
\end{align}
for some $M \in \N$.
For the norm ${L^1(|x| > 4)}$,
we also have by 
Lemmas \ref{Sj_s<1} \eqref{Sj_s<1_outside}
and \ref{Sj_Lp}
\begin{align} \label{L^1_outside}
\begin{split}
\big\|
S_j^1f\, S_j^2g
\big\|_{L^1(|x| > 4)}
\le
\| S_j^1f \|_{L^1(|x| > 4)}
\| S_j^2g \|_{L^\infty (\R^n)}
\lesssim
(2^j)^{\frac{sn}{2}}
\|\theta_1\|_{C^M}
\|\theta_2\|_{C^M}.
\end{split}
\end{align}
Thus, we shall consider the estimate of the second term in the right hand side of \eqref{decomp_first_kind}.
We decompose it as follows;
\begin{align} \label{decomL1_s<1}
\begin{split}
&
\big\|
S_j^1f\, S_j^2g
\big\|_{L^1(2r < |x| \le 4)}
\le
\sum_{k \in \N, \,\, 2^{k} r \le 4}
\big\|
S_j^1f\, S_j^2g
\big\|_{L^1(2^{k} r \le |x| \le 2^{k+1}r)}
\\
&
=
\Big(
\sum_{k \in \N, \,\, 2^{k} r < 2^{-j(1-s)}}
+
\sum_{k \in \N, \,\, 2^{-j(1-s)} \le 2^{k} r \le 4}
\Big)
\big\|
S_j^1f\, S_j^2g
\big\|_{L^1(2^{k} r \le |x| \le 2^{k+1}r)} .
\end{split}
\end{align}
Here,
we remark that
the first sum in the second line
vanishes if $2^{-j(1-s)} \le 2r$.

We show that the following estimate holds;
for $0 \le t \le 1$ and $0 \le N < \frac{n}{2(1-s)}$,
\begin{align} \label{SSsmall<1}
\begin{split}
&
\big\|
S_j^1f\, S_j^2g
\big\|_{L^1(2^k r \le |x| \le 2^{k+1}r)}
\\
&
\lesssim
2^{-jm} \,
\big( 2^{j(1-s)} 2^{k} r \big)^{ -N(1-t) + \frac{(1-s)n}{2}}
\min \big\{ (2^{j} r)^{t}, ( 2^{j} r )^{ -\frac{nt}{2} } \big\}
\|\theta_1\|_{C^M}
\|\theta_2\|_{C^M}
\end{split}
\end{align}
for $k \in \N$ satisfying that $2^{k} r \le 4$.
In fact, it follows from 
Lemmas \ref{Sj_L2} \eqref{Sj_L2_torus}
and \ref{Sj_s<1} \eqref{Sj_s<1_inside}
that
\begin{align*} 
&
\big\|
S_j^1f\, S_j^2g
\big\|_{L^1(2^k r \le |x| \le 2^{k+1}r)}
\le
\big\|
S_j^1f
\big\|_{L^2(2^k r \le |x| \le 2^{k+1}r)}
\big\|
S_j^2g
\big\|_{L^2(2^k r \le |x| \le 2^{k+1}r)}
\\
&
\lesssim
( 2^{j} )^{\frac{n}{2} } \,
\big( 2^{j(1-s)} 2^{k} r \big)^{ -N(1-t) }
\min \big\{ (2^{j} r)^{t}, ( 2^{j} r )^{ -\frac{nt}{2} } \big\} 
(2^k r)^{\frac{(1-s)n}{2}}
\|\theta_1\|_{C^M}
\|\theta_2\|_{C^M}
\|g\|_{L^\infty}
\\
&
=
( 2^{j} )^{\frac{n}{2} - \frac{(1-s)^{2} n}{2}} \,
\big( 2^{j(1-s)} 2^{k} r \big)^{ -N(1-t) + \frac{(1-s)n}{2}}
\min \big\{ (2^{j} r)^{t}, ( 2^{j} r )^{ -\frac{nt}{2} } \big\}
\|\theta_1\|_{C^M}
\|\theta_2\|_{C^M}
\\
&
=
2^{ -jm } \,
\big( 2^{j(1-s)} 2^{k} r \big)^{ -N(1-t) + \frac{(1-s)n}{2}}
\min \big\{ (2^{j} r)^{t}, ( 2^{j} r )^{ -\frac{nt}{2} } \big\}
\|\theta_1\|_{C^M}
\|\theta_2\|_{C^M},
\end{align*}
where we remark that
Lemmas \ref{Sj_L2} \eqref{Sj_L2_torus} and
\ref{Sj_s<1} \eqref{Sj_s<1_inside}
are 
applicable to the factor involved in $f$ and $g$, respectively,
since $k \in \N$ implies that $2^{k} r \ge 2r$
and also $2^{k} r \le 4$ implies that
$|x| \le 2^{k+1} r \le 8$.

The former sum in the second line of \eqref{decomL1_s<1} is 
estimated as follows.
Since $0 < s < 1$,
it follows from \eqref{SSsmall<1} 
with $0 \le t \le 1$ and $N=0$
that
\begin{align} \label{SSsmall<1_sum_1}
\begin{split}
&
\sum_{k \in \N, \, 2^k r < 2^{-j(1-s)}}
\big\|
S_j^1f\, S_j^2g
\big\|_{L^1(2^k r \le |x| \le 2^{k+1}r)}
\\
&
\lesssim
2^{ -jm }
\min \big\{ (2^{j} r)^{t} , ( 2^{j} r )^{ -\frac{nt}{2} } \big\} 
\|\theta_1\|_{C^M}
\|\theta_2\|_{C^M}
\sum_{ k \, : \, 2^{j(1-s)} 2^k r < 1 }
\big(
2^{j(1-s)} 2^k r 
\big)^{\frac{(1-s)n}{2}}
\\
&
\lesssim
2^{ -jm } 
\min \big\{ (2^{j} r)^{t} , ( 2^{j} r )^{ -\frac{nt}{2} } \big\} 
\|\theta_1\|_{C^M}
\|\theta_2\|_{C^M} .
\end{split}
\end{align}

On the other hand, 
the latter sum in the second line of \eqref{decomL1_s<1} is 
estimated as follows.
Since $0 < s < 1$, we have $\frac{(1-s)n}{2} < \frac{n}{2(1-s)}$,
and consequently we can choose
$0 < t < 1$ and $N > 0$ satisfying
$\frac{(1-s)n}{2(1-t)} < N < \frac{n}{2(1-s)}$.
Therefore,
by \eqref{SSsmall<1} with such $t$ and $N$,
it follows that 
\begin{align} \label{SSsmall<1_sum_2}
\begin{split}
&
\sum_{k \in \N, \, 2^{-j(1-s)} \le 2^k r \le 4}
\big\|
S_j^1f\, S_j^2g
\big\|_{L^1(2^k r \le |x| \le 2^{k+1}r)}
\\
&
\lesssim
2^{-jm}
\min \big\{ (2^{j} r)^{t}, ( 2^{j} r )^{ -\frac{nt}{2} } \big\} 
\|\theta_1\|_{C^M}
\|\theta_2\|_{C^M}
\sum_{ k \, : \, 2^{j(1-s)} 2^k r \ge 1 }
\big(
2^{j(1-s)} 2^k r
\big)^{-N(1-t) + \frac{(1-s)n}{2}}
\\
&
\lesssim
2^{-jm}
\min \big\{ ( 2^{j} r )^{t}, ( 2^{j} r )^{ -\frac{nt}{2} } \big\}
\|\theta_1\|_{C^M}
\|\theta_2\|_{C^M} .
\end{split}
\end{align}

Thus,
taking the same $0 < t < 1$ 
for \eqref{SSsmall<1_sum_1} and \eqref{SSsmall<1_sum_2}
and then
combining them with \eqref{decomL1_s<1}, we obtain
\begin{align*}
\big\|
S_j^1f\, S_j^2g
\big\|_{L^1(2r \le |x| \le 4)}
\lesssim
2^{ -jm } 
\min \big\{ ( 2^{j} r )^{t}, ( 2^{j} r )^{ -\frac{nt}{2} } \big\} \,
\|\theta_1\|_{C^M}
\|\theta_2\|_{C^M}.
\end{align*}
Hence,  combining this with \eqref{L^1_inside} and \eqref{L^1_outside}, 
we conclude that \eqref{goal} holds with 
$(U, V) = (S, S)$.
This completes the proof for the case $0 < s < 1$.


\subsection{Case $s>1$}
\label{s>1}
Before beginning with the proof,
let us recall the critical order for $s>1$:
\[
m = - \frac{ns}{2} .
\]
For this $m$,
in what follows,
we will prove that \eqref{goal} holds.
In this subsection,
we take a $j_{0} \in \N$ such that
$2^{j_{0} (s-1)} \ge 2$.
For such $j_{0} \in \N$,
we have by 
Lemma \ref{T} \eqref{T_Lq}
or \ref{Sj_Lp} 
\begin{align*}
\sum_{ 1 \le j \le j_{0} }
2^{jm} 
\big\| U_{j}^{1} f \,
V_{j}^{2} g \big\|_{L^1}
&
\lesssim 
\sum_{ 1 \le j \le j_{0} }
( 2^{j} )^{ m + ns } \,
\|\theta_1\|_{C^{M}}
\|\theta_2\|_{C^{M}}
\| f \|_{ L^1 } \| g \|_{ L^\infty }
\\
&
\lesssim 
\|\theta_1\|_{C^{M}}
\|\theta_2\|_{C^{M}}
\end{align*}
for $U,V \in \{S, T\}$.
Therefore, 
in order to achieve \eqref{goal},
it is sufficient to show that
\begin{equation} \label{goal_s>1}
\sum_{ j>j_{0} }
2^{jm} 
\big\| U_{j}^{1} f \,
V_{j}^{2} g \big\|_{L^1}
\lesssim 
\|\theta_1\|_{C^{M}}
\|\theta_2\|_{C^{M}} ,
\qquad
U,V \in \{S, T\} .
\end{equation}
To this end,
except for the case $(U,V) = (T, T)$,
we split the norm of $L^1(\R^n)$
as follows:
\begin{align} \label{decomL1}
\begin{split}
&
\big\| U_{j}^{1} f \,
V_{j}^{2} g \big\|_{L^1 (\R^n) }
=
\big\| U_{j}^{1} f \,
V_{j}^{2} g \big\|_{L^1 (|x| \le 2^{j(s-1)+1}) }
+
\big\| U_{j}^{1} f \,
V_{j}^{2} g \big\|_{L^1 (|x| \ge 2^{j(s-1)+1}) }
\\
&
\quad
\le
\big\| U_{j}^{1} f \,
V_{j}^{2} g \big\|_{L^1 (|x| \le 2^{j(s-1)+1}) }
+
\sum_{ k \in \N ,\, 2^{k} \ge 2^{j(s-1)} } 
\big\| U_{j}^{1} f \,
V_{j}^{2} g \big\|_{L^1 ( 2^{k} \le |x| \le 2^{k+1} ) } ,
\end{split}
\end{align}
where, the sum over 
$2^{k} \ge 2^{j(s-1)}$
should be read as 
the sum over $k \ge k_{0}$
with a positive integer $k_{0} = k_{0} (j)$ satisfying that
$2^{k_{0}-1} < 2^{j(s-1)} \le 2^{k_{0}}$.
Here, we are able to choose such $k_{0} \in \N$,
since $2^{j (s-1)} \ge 2$ for $j > j_{0}$.
Now, we shall prove \eqref{goal_s>1}.

\subsubsection{Estimate for
$T_{j}^{1} f \, T_{j}^{2} g$}

By 
Lemma \ref{T} \eqref{T_Lq},
\begin{align*}
\big\| T_{j}^{1} f \,
T_{j}^{2} g \big\|_{L^1}
&
\le
\| T_{j}^{1} f \|_{L^1}
\| T_{j}^{2} g \|_{L^\infty}
\\
&
\lesssim
\|\varphi(\cdot) \, \theta_{1} (2^{-j} \cdot)\|_{ C^{M} }
\|\varphi(\cdot) \, \theta_{2} (2^{-j} \cdot)\|_{ C^{M} }
\|f\|_{L^1} \|g\|_{L^\infty}
\\
&
\lesssim
\|\theta_{1}\|_{ C^{M} }
\|\theta_{2}\|_{ C^{M} } ,
\end{align*}
and thus,
since $m = - \frac{ns}{2} < 0$,
\eqref{goal_s>1} holds for $(U, V) = (T, T)$.

\subsubsection{Estimate for
$T_{j}^{1} f \, S_{j}^{2} g$}

We use the decomposition \eqref{decomL1}.
For the first term in \eqref{decomL1},
using 
Lemma \ref{T} \eqref{T_Lq} with $(p, q)=(1, 2)$ and 
Lemma \ref{Sj_s>1} with $A=2^{j(s-1) +1}$ 
to the factors for $f$ and $g$ respectively,
we have
\begin{align} \label{TSsmall}
\begin{split}
&
\big\| T_{j}^{1} f \,
S_{j}^{2} g \big\|_{L^1 (|x| \le 2^{j(s-1) +1}) }
\le
\| T_{j}^{1} f \|_{ L^2 (\R^n) }
\| S_{j}^{2} g \|_{L^2 (|x| \le 2^{j(s-1) + 1}) }
\\
&
\lesssim
\|\varphi(\cdot) \, \theta_{1} (2^{-j} \cdot)\|_{ C^{M} }
\|f\|_{L^1} \,
(2^{j(s-1)} )^{ \frac{n}{2} }
\|\theta_{2}\|_{ C^{M} }
\|g\|_{L^\infty}
\\
&
\lesssim
(2^{j} )^{ \frac{n(s-1)}{2} }
\|\theta_{1}\|_{ C^{M} }
\|\theta_{2}\|_{ C^{M} } .
\end{split}
\end{align}
Next,
for the summand in the sum of \eqref{decomL1},
we have 
by H\"older's inequality,
Lemma \ref{T} \eqref{T_Linfty} with $A = 2^k$,
and Lemma \ref{Sj_s>1} with $A = 2^{k+1}$
\begin{align*}
&
\big\| T_{j}^{1} f \,
S_{j}^{2} g \big\|_{L^1 ( 2^{k} \le |x| \le 2^{k+1} ) }
\lesssim
(2^{k})^{ \frac{n}{2} }
\big\| T_{j}^{1} f \big\|_{L^{\infty} ( 2^{k} \le |x| \le 2^{k+1} ) }
\big\|S_{j}^{2} g \big\|_{L^2 ( 2^{k} \le |x| \le 2^{k+1} ) }
\\
&
\lesssim
(2^{k})^{ \frac{n}{2} }
\cdot 
(2^{k})^{ -(n+s) }
\|\varphi(\cdot) \, \theta_{1} (2^{-j} \cdot)\|_{ C^{M} }
\cdot
(2^{k})^{ \frac{n}{2} }
\|\theta_{2}\|_{ C^{M} } 
\| g\|_{L^\infty}
\\
&
\lesssim
2^{-ks} \, 
\|\theta_{1}\|_{ C^{M} }
\|\theta_{2}\|_{ C^{M} } ,
\end{align*}
where,
it should be remarked that
Lemma \ref{Sj_s>1}
is applicable
to the factor involved in $g$,
since $2^{k+1} \ge 2^{j(s-1)}$
in the sum of \eqref{decomL1}.
This yields 
from $s > 1$ 
that
\begin{align} \label{TSlarge}
\begin{split}
\sum_{ k \in \N ,\, 2^{k} \ge 2^{j(s-1)} } 
\big\| T_{j}^{1} f \,
S_{j}^{2} g \big\|_{L^1 ( 2^{k} \le |x| \le 2^{k+1} ) }
&
\lesssim
\|\theta_{1}\|_{ C^{M} }
\|\theta_{2}\|_{ C^{M} }
\sum_{ k \in \N } 
2^{-ks}
\approx
\|\theta_{1}\|_{ C^{M} }
\|\theta_{2}\|_{ C^{M} } .
\end{split}
\end{align}
Therefore,
combining \eqref{TSsmall} and \eqref{TSlarge} with \eqref{decomL1}
we obtain
\[
\big\| T_{j}^{1} f \,
S_{j}^{2} g \big\|_{L^1}
\lesssim
(2^{j} )^{ \frac{n(s-1)}{2} }
\|\theta_{1}\|_{ C^{M} }
\|\theta_{2}\|_{ C^{M} },
\]
which implies that \eqref{goal_s>1} holds for $(U, V) = (T, S)$
since $m < - \frac{n(s-1)}{2}$.

\subsubsection{Estimate for
$S_{j}^{1} f \, T_{j}^{2} g$}

For the first term in \eqref{decomL1},
Lemmas \ref{Sj_L2} \eqref{Sj_L2_Rn}
and \ref{T} \eqref{T_Lq}
yield that,
for $0 \le t \le 1$,
\begin{align} \label{STsmall}
\begin{split}
\| S_{j}^{1} f \,
T_{j}^{2} g 
\|_{L^1 (|x| \le 2^{j(s-1)+1}) }
&
\lesssim
( 2^{j} )^{ \frac{n(s-1) }{2} }
\| S_{j}^{1} f \|_{ L^2 (\R^n) }
\| T_{j}^{2} g \|_{ L^\infty (\R^n) }
\\
&
\lesssim
( 2^{j} )^{ \frac{ns }{2} }
\min \big\{ (2^{j} r)^{t} , ( 2^{j} r )^{ -\frac{nt}{2} } \big\}
\| \theta_1 \|_{ C^{0} }
\| \theta_2 \|_{ C^{M} } .
\end{split}
\end{align}
We shall next consider the sum of \eqref{decomL1}.
By H\"older's inequality and
Lemmas \ref{Sj_L2} \eqref{Sj_L2_torus}
and \ref{T} \eqref{T_Lq},
the summand in \eqref{decomL1}
is estimated by
\begin{align*}
&
\| S_{j}^{1} f \,
T_{j}^{2} g 
\|_{L^1 ( 2^{k} \le |x| \le 2^{k+1} ) } 
\lesssim
(2^{k})^{ \frac{n}{2} }
\| S_{j}^{1} f \|_{ L^2 ( 2^{k} \le |x| \le 2^{k+1} ) } 
\| T_{j}^{2} g \|_{ L^{\infty} ( \R^n ) } 
\\
&
\lesssim
(2^{k})^{ \frac{n}{2} }
\cdot
(2^{j})^{ \frac{n}{2} } \,
\big( 2^{j(1-s)} 2^{k} \big)^{ -N(1-t) }
\min \big\{ (2^{j} r)^{t}, ( 2^{j} r )^{ -\frac{nt}{2} } \big\} 
\| \theta_1 \|_{ C^{M} }
\| \theta_2 \|_{ C^{M} } 
\\
&
=
( 2^{j} )^{ \frac{ns}{2} } \,
\big( 2^{j(1-s)} 2^{k} \big)^{ -N(1-t) + \frac{n}{2} }
\min \big\{ (2^{j} r)^{t}, ( 2^{j} r )^{ -\frac{nt}{2} } \big\} 
\| \theta_1 \|_{ C^{M} }
\| \theta_2 \|_{ C^{M} } ,
\end{align*}
where
we notice that
$2^{k} \ge 2r$ holds
in the sum of \eqref{decomL1},
since
$k$ is restricted to $\N$ and
$r \le 1$:
this allows us to apply Lemma \ref{Sj_L2} \eqref{Sj_L2_torus} 
with $A=2^{k}$
to the factor with respect to $f$.
Then,
choosing
$0 < t < 1$ and $N > 0$ 
such that 
$-N(1-t) + \frac{n}{2} < 0$,
we have
\begin{align} \label{STlarge}
\begin{split}
&
\sum_{ k \in \N ,\, 2^{k} \ge 2^{j(s-1)} } 
\| S_{j}^{1} f \,
T_{j}^{2} g 
\|_{L^1 ( 2^{k} \le |x| \le 2^{k+1} ) } 
\\
&
\lesssim
( 2^{j} )^{ \frac{ns}{2} } 
\min \big\{ (2^{j} r)^{t}, ( 2^{j} r )^{ -\frac{nt}{2} } \big\} 
\| \theta_1 \|_{ C^{M} }
\| \theta_2 \|_{ C^{M} } 
\sum_{k \, : \, 2^{j(1-s)} 2^{k} \ge 1 } 
\big( 2^{j(1-s)} 2^{k} \big)^{ -N(1-t) + \frac{n}{2} }
\\
&
\approx
( 2^{j} )^{ \frac{ns}{2} } 
\min \big\{ (2^{j} r)^{t}, ( 2^{j} r )^{ -\frac{nt}{2} } \big\} 
\| \theta_1 \|_{ C^{M} }
\| \theta_2 \|_{ C^{M} } .
\end{split}
\end{align}
Thus, 
by choosing the same $0<t<1$ 
for \eqref{STsmall} and \eqref{STlarge},
and
by
combining them with \eqref{decomL1},
we obtain 
for such $0<t<1$ 
\begin{align*}
&
\big\| S_{j}^{1} f \,
T_{j}^{2} g \big\|_{L^1}
\lesssim
( 2^{j} )^{ \frac{ns}{2} } \,
\min \big\{ (2^{j} r)^{t}, ( 2^{j} r )^{ -\frac{nt}{2} } \big\} \,
\|\theta_{1}\|_{ C^{M} }
\|\theta_{2}\|_{ C^{M} }
,
\end{align*}
which implies that \eqref{goal_s>1} holds for $(U, V) = (S, T)$.

\subsubsection{Estimate for
$S_{j}^{1} f \, S_{j}^{2} g$}

For the first term in \eqref{decomL1}, by 
Lemmas \ref{Sj_L2} \eqref{Sj_L2_Rn}
and \ref{Sj_s>1},
we have for $0 \le t \le 1$
\begin{align} \label{SSsmall}
\begin{split}
&
\| S_{j}^{1} f \,
S_{j}^{2} g 
\|_{L^1 (|x| \le 2^{j(s-1)+1}) }
\lesssim
\| S_{j}^{1} f \|_{ L^2 (\R^n) }
\| S_{j}^{2} g \|_{ L^2 (|x| \le 2^{j(s-1)+1}) }
\\
&
\lesssim
( 2^{j} )^{ \frac{n}{2} }
\min \big\{ (2^{j} r)^{t}, ( 2^{j} r )^{ -\frac{nt}{2} } \big\} \,
\| \theta_1 \|_{ C^{M} }
\cdot
( 2^{j} )^{ \frac{n(s-1)}{2} }
\| \theta_2 \|_{ C^{M} } 
\\
&
=
( 2^{j} )^{ \frac{ns}{2} }
\min \big\{ (2^{j} r)^{t}, ( 2^{j} r )^{ -\frac{nt}{2} } \big\} \,
\| \theta_1 \|_{ C^{M} }
\| \theta_2 \|_{ C^{M} } .
\end{split}
\end{align}
For the summand of \eqref{decomL1},
we have by 
Lemmas \ref{Sj_L2} \eqref{Sj_L2_torus}
and \ref{Sj_s>1} 
\begin{align*}
&
\| S_{j}^{1} f \,
S_{j}^{2} g 
\|_{L^1 ( 2^{k} \le |x| \le 2^{k+1} ) } 
\le
\| S_{j}^{1} f \|_{ L^2 ( 2^{k} \le |x| \le 2^{k+1} ) } 
\| S_{j}^{2} g \|_{ L^2 ( 2^{k} \le |x| \le 2^{k+1} ) } 
\\
&
\lesssim
(2^{j})^{ \frac{n}{2} } \,
\big( 2^{j(1-s)} 2^{k} \big)^{ -N(1-t) }
\min \big\{ (2^{j} r)^{t}, ( 2^{j} r )^{ -\frac{nt}{2} } \big\} 
\| \theta_1 \|_{ C^{M} }
\cdot
(2^{k})^{ \frac{n}{2} }
\| \theta_2 \|_{ C^{M} } 
\\
&
=
(2^{j})^{ \frac{ns}{2} } 
\big( 2^{j(1-s)} 2^{k} \big)^{ -N(1-t) + \frac{n}{2} }
\min \big\{ (2^{j} r)^{t}, ( 2^{j} r )^{ -\frac{nt}{2} } \big\} \,
\| \theta_1 \|_{ C^{M} }
\| \theta_2 \|_{ C^{M} } ,
\end{align*}
which yields that,
since there exist
$N>0$ and $0<t<1$ 
such that $-N(1-t) + \frac{n}{2} < 0$,
\begin{align} \label{SSlarge}
\begin{split}
&
\sum_{ k \in \N ,\, 2^{k} \ge 2^{j(s-1)} } 
\| S_{j}^{1} f \,
S_{j}^{2} g 
\|_{L^1 ( 2^{k} \le |x| \le 2^{k+1} ) } 
\lesssim
(2^{j})^{ \frac{ns}{2} } 
\min \big\{ (2^{j} r)^{t}, ( 2^{j} r )^{ -\frac{nt}{2} } \big\} \,
\| \theta_1 \|_{ C^{M} }
\| \theta_2 \|_{ C^{M} } .
\end{split}
\end{align}
Therefore,
choosing the same $0<t<1$ 
for \eqref{SSsmall} and \eqref{SSlarge}
and then
combining them with \eqref{decomL1},
we obtain 
for such $0<t<1$ 
that
\begin{align*}
&
\| S_{j}^{1} f \,
S_{j}^{2} g \|_{L^1 }
\lesssim
( 2^{j} )^{ \frac{ns}{2} } 
\min \big\{ (2^{j} r)^{t}, ( 2^{j} r )^{ -\frac{nt}{2} } \big\} \,
\|\theta_{1}\|_{ C^{M} }
\|\theta_{2}\|_{ C^{M} } .
\end{align*}
This gives \eqref{goal_s>1} for $(U, V) = (S, S)$,
and we complete the proof for $s > 1$.

\section{Necessary conditions on $m$}
\label{Nec_on_m}

In this section, we shall give a proof of Theorem \ref{thm_necessity}.
In this section, we use the notation 
$X_r$ given in \eqref{funcspXr}.

\begin{proof}[Proof of Theorem \ref{thm_necessity} ]
Let $0 < s < 1$ or $1 < s < \infty$,
and let
$m \in \R$, 
$(p, q) \in \RomI \cup \II \cup \IV \cup \VI$, and $1/ r = 1/p + 1/q$.
If all bilinear operators $T^{s}_{\sigma}$, 
$\sigma \in S^m_{1,0}(\R^{2n})$,
are bounded from $H^p \times H^q$ to $X_r$, 
then, 
by virtue of the closed graph theorem,
there exist $c > 0$ and $N \in \N$
such that 
\begin{align} \label{nec_on_T}
&
\big\|
T_{\sigma}^{s} 
\big\|_{H^p \times H^q \to X_r}
\le
c
\,
\max_{|\alpha|, |\beta| \le N}
\big\|
(1+|\xi| + |\eta|)^{-m + |\alpha| + |\beta|}
\partial_{\xi}^{\alpha}
\partial_{\eta}^{\beta}
\sigma(\xi, \eta)
\big\|_{L^\infty(\R^{2n})}
\end{align}
 holds for all $\sigma \in S^{m}_{1, 0}(\R^n)$
(see B\'enyi--Bernicot--Maldonado--Naibo--Torres \cite[Lemma 2.6]{BBMNT-IUMJ}
for the argument using the closed graph theorem). 

Now, we take two functions $ \theta $ and $\phi$ 
such that
\begin{align*}
&
\theta \in C^\infty_0(\R^n),
\quad
\supp \theta \subset \{3^{-1} \le |\xi| \le 3\},
\quad
\theta(\xi) = 1 
\,\,\,\,
\text{on}
\,\,\,\, 
\{2^{-1} \le |\xi| \le 2\},
\\
&
\phi \in C^\infty_0(\R^n),
\quad
\supp \phi \subset \{ |\xi| \le 3\},
\quad
\phi(\xi) = 1 
\,\,\,\,
\text{on}
\,\,\,\, 
\{ |\xi| \le 2\}.
\end{align*}
For $j \in \N$, we set
\begin{align*}
\sigma_j(\xi, \eta)
=
2^{jm}
\theta(2^{-j}\xi)
\phi(2^{-j}\eta).
\end{align*}
Then we have
\begin{align*}
\big|
\partial_{ \xi }^{ \alpha }
\partial_{ \eta }^{ \beta }
\sigma_j(\xi, \eta)
\big|
\le
C_{ \alpha, \beta }
(1 + |\xi| + |\eta|)^{ m-|\alpha|-|\beta| },
\quad
\alpha, \beta \in \N_0^n,
\end{align*}
uniformly in $j \in \N$.
Hence,  
by \eqref{nec_on_T}, 
we see that 
there exists $C > 0$ such that
\begin{align} \label{est_opnorm}
\big\|
T_{\sigma_j}^{s}
\big\|_{H^{p} \times H^{q} \to X_r}
\le
C,
\,\,\,\,
j \in \N.
\end{align}
We shall prove that 
\eqref{est_opnorm} holds only if $m \le m_s(p, q)$.

Take a function 
$\psi \in C^\infty_0(\R^n)$ such that 
$\supp \psi \subset \{2^{-1} \le |\xi| \le 2\}$
and
$\psi (\xi)\neq 0$ on $\{2/3 \le |\xi| \le 3/2\}$,
and set
\begin{align*}
&
f^{+}_{j} (x)
=
\Big(
e^{i|\xi|^s}
\psi(2^{-j}\xi)
\Big)^{\vee} (x),
\\
&
f^{-}_{j} (x)
=
\Big(
e^{-i|\xi|^s}
\psi(2^{-j}\xi)
\Big)^{\vee} (x),
\\
&
f_{j}(x)
=
\Big(
\psi(2^{-j}\xi)
\Big)^{\vee} (x)
=
2^{jn}
(\psi)^{\vee}(2^j x).
\end{align*}
Then we have the following estimates;
\begin{align} 
&
\|f^{\pm}_{j}\|_{H^p}
\approx
\|f^{\pm}_{j}\|_{L^p}
\approx
2^{j(n - \frac{sn}{2})}
2^{-j(1-s)\frac{n}{p}}
\,\,\,\,
\text{for} 
\,\,\,\,
1 \le p \le \infty,
\label{est_of_f_jpm}
\\
&
\|f_{j}\|_{H^p}
\approx
\|f_{j}\|_{L^p}
\approx
2^{j (n - \frac{n}{p})}
\label{est_of_f_j1}
\,\,\,\,
\text{for} 
\,\,\,\,
1 \le p \le \infty,
\\
&
\|(f_{j})^2\|_{BMO}
\approx
2^{2 j n}
\label{est_of_f_j2}.
\end{align}
In fact, since the Fourier transform of $f^{\pm}_j$ 
are supported in the annulus $\{2^{j-1} \le |\xi| \le 2^{j+1}\}$,
we have the first inequality in \eqref{est_of_f_jpm},
and the second $\approx$ in \eqref{est_of_f_jpm} 
follows from Proposition \ref{Kj_psi}.
On the other hand, by a straightforward calculation,
we see that \eqref{est_of_f_j1} and \eqref{est_of_f_j2} hold.

\bigskip
\noindent
\textit{Proof of the case $(p, q) \in \RomI$.}
In this case we use the function $f^{-}_{j}$.
Since 
$\psi(2^{-j} \xi) \theta(2^{-j} \xi) 
= 
\psi(2^{-j} \xi) $
and
$\psi(2^{-j} \eta) \phi(2^{-j} \eta) 
= 
\psi(2^{-j} \eta) $, 
we have
\begin{align*}
T^{s}_{\sigma_j}(f^{-}_j, f^{-}_j)(x)
=
2^{jm}
\big(
f_{j}(x)
\big)^2,
\end{align*}
and hence, we obtain
\begin{align*}
\big\|
T^{s}_{\sigma_j}(f^{-}_{j}, f^{-}_{j})
\big\|_{X_r}
=
2^{jm}
\big\|
( f_{j} )^2
\big\|_{X_r}
\approx
2^{j(m + 2n - \frac{n}{r})},
\quad
j \in \N,
\end{align*}
where, in the last inequality, we used \eqref{est_of_f_j1} combined with
the identity $\|( f_{j} )^2\|_{L^r} = \|f_{j}\|_{L^{2r}}^2$
if $r < \infty$, and
used \eqref{est_of_f_j2} 
if $r = \infty$.
Combining this 
with \eqref{est_opnorm} and \eqref{est_of_f_jpm}, 
we obtain
\begin{align*}
2^{j(m + 2n - \frac{n}{r})}
\lesssim
2^{j(n - \frac{sn}{2})}
2^{-j(1-s)\frac{n}{p}}
\,
2^{j(n - \frac{sn}{2})}
2^{-j(1-s)\frac{n}{q}},
\,\,\,\,
j \in \N.
\end{align*}
This is possible only when 
$m 
\le 
-sn(\frac{1}{2} - \frac{1}{p} 
+
\frac{1}{2} - \frac{1}{q}
)$.
This completes the  proof of the case $(p, q) \in \RomI$.

\bigskip
\noindent
\textit{Proof of the case $(p, q) \in \II$.}
By the same reasons as above,
we have 
\begin{align*}
T^{s}_{\sigma_j}(f_j, f_j)(x)
=
2^{jm}
(f^{+}_{j}(x))^2,
\end{align*}
and thus \eqref{est_of_f_jpm} implies
that, since $2r \ge 1$ for $(p, q) \in \II$,

\begin{align*}
\|
T^{s}_{\sigma_j}(f_j, f_j)
\|_{L^r}
=
2^{jm}
\big\|
(f^{+}_{j})^2
\big\|_{L^{r}}
=
2^{jm}
\big\|
f^{+}_{j}
\big\|_{L^{2r}}^2
\approx
2^{jm}
\big(
2^{j(n - \frac{sn}{2})}
\big)^2
2^{-j(1-s) \frac{n}{r}}.
\end{align*}
Hence, 
if \eqref{est_opnorm} holds, then
this estimate and \eqref{est_of_f_j1} imply that
\begin{align*}
2^{jm}
\big(
2^{j(n - \frac{sn}{2})}
\big)^2
\,
2^{-j(1-s) \frac{n}{r}}
\lesssim
2^{j (n - \frac{n}{p})}
2^{j (n - \frac{n}{q})},
\,\,\,\,
j \in \N,
\end{align*}
which holds only if
$m
\le
-sn
(\frac{1}{p} - \frac{1}{2}
+
\frac{1}{q} - \frac{1}{2})$.
This completes the proof of the case $(p, q) \in \II$.

\bigskip
\noindent
\textit{Proof of the case $(p, q) \in \IV$.}
In addition to the functions $f_j^{\pm}$ and $f_j$, 
we use the following functions;
\begin{align*}
&
g_j(x)
= 
\Big(
e^{-i|\eta|^s}
\psi(2^{-j(1-s)} \eta)
\Big)^{\vee}(x),
\\
&
h_j(x)
=
\Big(
e^{-i2|\eta|^s}
\psi(2^{-j}\eta)
\Big)^{\vee}(x),
\end{align*}
where the function $\psi$ is the same given in the definition of $f^{\pm}_j$ and $f_j$.
Since the support of 
$\widehat{g}_j$ is 
included in the annulus
$\{2^{j(1-s) -1} \le |\eta| \le 2^{j(1-s) + 1}\}$,
if $0 < s < 1$,
then we have by Proposition \ref{Kj_psi}
\begin{align} \label{est_of_g_j}
\|g_j\|_{H^q}
\approx
\|g_j\|_{L^q}
\approx
2^{j(1-s)(n - \frac{sn}{2})} 
2^{-j(1-s)^2 \frac{n}{q}}
\,\,\,\,
\text{for} 
\,\,\,\,
1 \le q \le \infty
\end{align}
(we notice that this holds when $s<1$,
since
Proposition \ref{Kj_psi} treats the function 
whose Fourier support locates far from the origin;
the support of $\widehat{g}_j$ locates near the origin when $s>1$).
On the other hand, 
in the same way as for the functions $f^{\pm}_j$, 
we also have 
\begin{align} \label{est_of_h_j}
\|h_j\|_{H^q}
\approx
\|h_j\|_{L^q}
\approx
2^{j(n-\frac{sn}{2})} 2^{-j(1-s)\frac{n}{q}}
\,\,\,\,
\text{for} 
\,\,\,\,
1 \le q \le \infty.
\end{align}

Now, we first consider the case $0< s < 1$.
Observe that
$\theta(2^{-j} \xi) \psi(2^{-j} \xi) = \psi(2^{-j} \xi)$ and
$\psi(2^{-j(1-s)} \eta) \phi(2^{-j} \eta) = \psi(2^{-j(1-s)} \eta)$.
Hence we have
\begin{align*}
T^{s}_{\sigma_j}(f_j, g_j)(x)
&
=
2^{jm}
2^{j(1-s)n}
f^{+}_{j}(x)
( \psi )^{\vee}
(2^{j(1-s)}x).
\end{align*}
Then, 
from \eqref{Kjxradial},
it holds that 
\begin{align*}
\big|
T^{s}_{\sigma_j}(f_{j}, g_j)(x)
\big|
\ichi_{\{a^{\prime} < 2^{j(1-s)}|x| \le b^{\prime}\}}
&
\approx
2^{jm}
2^{j(1-s)n}
2^{j(n-\frac{ns}{2})}
|
(\psi)^{\vee}(2^{j(1-s)}x)
|
\ichi_{\{a^{\prime} < 2^{j(1-s)}|x| \le b^{\prime}\}}
\end{align*}
for all $j > j_0$, 
where $a^{\prime}$, $b^{\prime}$ and $j_0$ are the same 
given in Proposition \ref{Kj_psi} (see \eqref{Kjxradial}).
Thus, we obtain 
\begin{align*}
\big\|
T^{s}_{\sigma_j}(f_{j}, g_j)
\big\|_{L^r}
&
\gtrsim
2^{jm}
2^{j(n-\frac{ns}{2})}
2^{j(1-s)n}
\big\|
(\psi)^{\vee}(2^{j(1-s)}x)
\ichi_{\{a^{\prime} < 2^{j(1-s)}|x| \le b^{\prime}\}}
\big\|_{L^r_x}
\\
&
=
c
\,
2^{jm}
2^{j(n-\frac{ns}{2})}
2^{j(1-s)n}
2^{-j(1-s) \frac{n}{r}},
\quad
j > j_0, 
\end{align*}
with $c = \|(\psi)^{\vee} \ichi_{\{a^{\prime} < |x| \le b^{\prime}\}}\|_{L^{r}} > 0$.
Hence
it follows from 
\eqref{est_opnorm}, \eqref{est_of_f_j1} and \eqref{est_of_g_j}
that
\begin{align*}
2^{jm}
2^{j(n-\frac{ns}{2})}
2^{j(1-s)n}
2^{-j(1-s) \frac{n}{r}}
\lesssim
2^{j(n - \frac{n}{p})}
2^{j(1-s)(n - \frac{sn}{2})} 
2^{-j(1-s)^2 \frac{n}{q}},
\,\,\,\,
j > j_0,
\end{align*}
which is possible only when
$
m
\le
-sn
(\frac{1}{p}
-
\frac{1}{2})
-s(1-s)n
(\frac{1}{2}
-
\frac{1}{q})$.

Finally, we shall consider the case $s > 1$.
Since
\begin{align*}
T^{s}_{\sigma_j}(f_j, h_j)(x)
=
2^{jm}
f^{+}_{j}(x)
f^{-}_{j}(x),
\end{align*}
it follows from \eqref{Kjxradial} that 
\begin{align*}
\big|
T^{s}_{\sigma_j}(f_{j}, h_j)(x)
\big|
\ichi_{\{a^{\prime} < 2^{j(1-s)}|x| \le b^{\prime}\}}
&
\approx
2^{jm}
2^{j(n-\frac{ns}{2})}
2^{j(n-\frac{ns}{2})}
\ichi_{\{a^{\prime} < 2^{j(1-s)}|x| \le b^{\prime}\}}
\end{align*}
for all $j > j_0$.
Hence, we obtain
\begin{align*}
\big\|
T^{s}_{\sigma_j}(f_{j}, h_j)
\big\|_{L^r}
\gtrsim
2^{jm}
2^{j(n-\frac{ns}{2})}
2^{j(n-\frac{ns}{2})}
2^{-j(1-s)\frac{n}{r}},
\quad
j > j_0.
\end{align*}
Combining this with \eqref{est_opnorm}, \eqref{est_of_f_j1} and \eqref{est_of_h_j},
we have
\begin{align*}
2^{jm}
2^{j(n-\frac{sn}{2})}
2^{j(n-\frac{sn}{2})}
2^{-j(1-s)\frac{n}{r}}
\lesssim
2^{j(n - \frac{n}{p})}
2^{j(n - \frac{sn}{2})}
2^{-j(1-s)\frac{n}{q}},
\,\,\,\,
j > j_0,
\end{align*}
which is possible only when
$m
\le
-sn
(\frac{1}{p}
-
\frac{1}{2})$.
This completes the proof of the case $(p, q) \in \IV$.

\bigskip
\noindent
\textit{Proof of the case $(p, q) \in \VI$.}
Since the situation is symmetrical, 
we obtain the desired conclusion 
in the same way as for the case $(p, q) \in \IV$.
Thus Theorem \ref{thm_necessity}
is proved.
\end{proof}



\end{document}